\documentclass{amsart}
\usepackage{amssymb}
\usepackage{hyperref}
\usepackage{txfonts}
\usepackage{amsfonts}
\usepackage{hyperref}
\usepackage{amsfonts}
\usepackage{mathrsfs}

\newtheorem{theorem}{Theorem}[section]
\newtheorem{lemma}[theorem]{Lemma}
\newtheorem{proposition}[theorem]{Proposition}
\newtheorem{corollary}[theorem]{Corollary}

\theoremstyle{definition}
\newtheorem{definition}[theorem]{Definition}
\newtheorem{example}[theorem]{Example}

\theoremstyle{remark}
\newtheorem{remark}[theorem]{Remark}

\numberwithin{equation}{section}

\begin{document}

\title{On $f$-bi-harmonic maps between Riemannian manifolds }
\author{Wei-Jun Lu}
\address{Center of Mathematical Sciences, Zhejiang University,
\newline\indent Hangzhou, Zhejiang, 310027, P. R. China;
\newline\indent College of Mathematics and Computer Science,
Guangxi University for Nationalities,
\newline\indent Naning, Guangxi, 530006, P. R. China.\\
\newline\indent E-mail: weijunlu2008@126.com }

\date{}



\begin{abstract}
{\footnotesize  Both bi-harmonic map and $f$-harmonic map have nice
physical motivation and applications. In this paper, by combination
of these two harmonic maps, we introduce and study $f$-bi-harmonic
maps as the critical points of the $f$-bi-energy functional
$\frac{1}{2}\int_M f|\tau (\phi)|^2dv_{g}$. This class of maps
generalizes both concepts of harmonic maps and bi-harmonic maps. We
first derive the $f$-biharmonic map equation and then use it to
study $f$-bi-harmonicity of some special maps, including conformal
maps between  manifolds of same dimensions, some product maps
between direct product manifold and singly warped product manifold,
some projection maps from and some inclusion maps into a warped
product manifold.}
\end{abstract}
\maketitle

\section{Introduction}

First motivated by the physical interpretation of $f$-harmonic map
(see \cite{Ou1},\cite{LW}, \cite{RV}), we borrowed from the method
for studying bi-harmonic maps in \cite {PK,BMO} to investigate the
behaviors of $f$-harmonic maps from or into doubly warped product
manifold (WPM). We derived some characteristic equations for
$f$-harmonicity and also constructed some examples \cite{Lu1}.
Subsequently, we found that $f$-tension field don't involve the
Riemannian curvature tensor $\bar R$ on WPM unlike \cite{PK} and
\cite{BMO}.

To make amends for this shortcoming, we wanted to formulate a new
type of tension field which contains the Rimannian curvature
component like bi-tension field. Naturally, we focused on
constructing a field so-called bi-$f$-tension field or
$f$-bi-tension field via combining $f$-tension field and bi-tension
field. Thus we attempted to derive the Euler-Lagrangian equation by
the first variation for corresponding energy functional
$\frac{1}{2}\int_M f|\tau(\phi)|^2dv_g$ or $\frac{1}{2}\int_M
|\tau_f(\phi)|^2dv_g$ according to the canonical methods as same as
bi-energy functional and $f$-energy functional.

At that time,since the deduction is very complicated, together with
our poor processing techniques, we had attacked this problem vainly
for two weeks. At the very moment, Ou sent us the scan PDF file of
Ouakkas-Nasri-Djaa's article \cite{OND}. Although the terminology
about $f$-bi-tension field in \cite{OND} is not the terminology we
expected (from now on, we change it as bi-$f$-tension field), we
could directly use the already bi-$f$-tension field deduced by them
to discuss bi-$f$-harmonic maps whose domain or codomain is doubly
WPM (see \cite{Lu2}). With much more complicated and tedious
computations, together without any non-trivial example for
bi-$f$-harmonic map, the referee gave an unfavorable review of the
paper \cite{Lu2} and so far we dared not submit it to the Journal.

During on writing our Ph.D thesis (\cite{Lu3}) recently, we
rearranged the the paper \cite{Lu2}. In order to obtain some simpler
and interesting results, we made a modification by exchanging doubly
WPM to singly WPM. Thus we arrived at our original goal. Based on
this work, we again developed another new tension field so-called
$f$-bi-tension field (different from the terminology in \cite{OND})
which comes from the first variation of energy functional
$$\frac{1}{2}\int_M f|\tau(\phi)|^2dv_g.$$
Subsequently, we employed the method of \cite{BMO,Lu1} to discuss
the behaviors of $f$-bi-harmonic maps from or into singly WPM.

Just noted that the progress in the topic on $f$-bi-harmonic map
such as \cite{CET,Ch1,Ch2}, together with \cite{BFO}, we further
consider  $f$-bi-harmonic maps with conformal dilation.

In this paper, our main results are listed below:

 (i) $f$-B i-tension field $\tau_{2,f}(\phi)$ by the first
variation of $f$-bi-energy functional $\frac{1}{2}\int_M
f|\tau(\phi)|^2dv_g$, attached to Propositions \ref{5-11-3} and
\ref{5-12-1};

 (ii) Results on $f$-bi-harmonic maps with conformal dilation,
attached to Propositions \ref{5-14-01} and \ref{5-15-1};

 (iii) Characteristic behaviors of bi-$f$-harmonic maps from or
into singly WPM, attached to Theorem \ref{5-18-0}, Corollaries
\ref{5-19-6} and \ref{5-19-8}, Propositions \ref{4.2.4-1} and
\ref{4.2.4-9};

(iv) Characteristic behaviors of $f$-bi-harmonic maps from or into
singly WPM, attached to Corollaries \ref{5-19-12a} and
\ref{5-19-12b}, Proposition \ref{4.2.3-1ab}, Propositions
\ref{5-20-2} and \ref{5-20-5},  Propositions \ref{5-20-6} and
\ref{5-21-1}.

   The organization of this paper is as follows. In the second section is preliminary, which
reviews some basic definitions on biharmonic maps, $f$-harmonic
maps,  also gives the definition of doubly/singly WPMs and the more
explicit expressions of the connection $\bar\nabla$ and curvature
tensor $\bar R$ on singly WPM. Section 3 is devoted to briefly
recall the first variation of bi-$f$-harmonic map, the formula of
bi-$f$-tension field and bi-$f$-harmonic map, also includes the
results for bi-$f$-harmonic maps with conformal dilation and some
examples. In section 4, we deduce $f$-bi-tension field from the
corresponding energy functional and discuss $f$-bi-harmonic maps
with conformal dilation. In Section 5, we discuss the behaviors of
bi-$f$-harmonic maps whose domain or codomain is singly WPM and also
construct some non-trivial example. Section 6 is devoted to discuss
the behaviors of $f$-bi-harmonic maps from or into singly WPM and
also construct some non-trivial example. In the last section, we
make some comparisons on $f$-bi-harmonic map and bi-$f$-harmonic
map, also present a prising expectation.

\section{Preliminaries}

\subsection{Harmonic, bi-harmonic and $f$-harmonic maps}

Recall that the energy of a smooth map $\phi: (M,g) \to (N,h) $
between two Riemannian manifolds is defined by integral
$E(\phi)=\int_\Omega e(\phi) dv_g$, for every compact domain $\Omega
\subset M$ where $e(\phi)=\frac{1}{2}|d\phi|^2$ is energy density
and $\phi$ is called {\em harmonic} if it's a critical point of
energy. From the first variation formula for the energy, the
Euler-Lagrange equation is given by the vanishing of the {\em
tension field} $\tau(\phi)=\rm{Tr}_g \nabla d\phi$ (see \cite{ES}).
As the generalizations of harmonic maps, we now recall the concepts
of bi-harmonic maps and $f$-harmonic maps.

\begin{definition}
\par (i) Bi-harmonic maps $\phi: (M,g) \to (N,h) $ between Riemannian manifolds are critical
points of the bienergy functional
                    $$E_2(\phi)=\frac{1}{2}\int_\Omega |\tau(\phi)|^2 dv_g,$$
for any compact domain $\Omega \subset M$.
\par (ii)  An $f$-harmonic map with a positive function $f\in C^\infty (M)$ is a critical point of $f-$energy
   \begin{equation} \label{7.11-1}
     E_f(\phi)=\frac{1}{2}\int_{\Omega \Subset M } f |d(\phi)|^2 dv_g.
     \end{equation}
\end{definition}

The Euler-Lagrange equations give the bitension field
$\tau^2(\phi)$( \cite{Ji}) and the $f$-tension field equation
$\tau_f (\phi)$ (see \cite{Co}, \cite{OND},\cite{Ou1}),
respectively,
 \begin{equation} \label{1.1} \begin{array}{ll}
  &\tau_2(\phi) = -\rm{Tr}_g\big(\nabla_{.}^\phi\nabla_{.}^\phi \tau(\phi)
               -\nabla^\phi_{ ^M\!\nabla_{.}\ .}\tau(\phi)\big)
                 -\rm{Tr}_g (R^N(d\phi,\tau(\phi)) d\phi) )=0, \\
  &\tau_f (\phi)= f\tau(\phi)+d\phi(grad\ f)=0.
  \end{array}
  \end{equation}

\subsection{Connection and Riemannian curvature tensor on singly WPM}

First we refer to \cite{Un} and give the definition of doubly/singly
WPM.

\begin{definition} \rm  \label{7.12-1}
Let $(M,g)$ and $(N,h)$ be Riemannian manifolds of dimensions $m$
and $n$ respectively and let $\lambda : M \to (0, +\infty)$ and
$\mu: N \to (0,+\infty)$ be smooth functions. A {\it doubly warped
product manifold (WPM)$G = M \times_{(\mu,\lambda)} N$ }is the
product manifold $M \times N $ endowed with the doubly warped
product metric
  $\bar{g} =\mu^2  g \oplus \lambda^2 h$
defined by
  $$\bar{g} (X,Y) = (\mu\circ\pi_1)^2 g(d\pi_1(X), d\pi_1(Y ))
         + (\lambda\circ\pi_2)^2 h (d\pi_2(X), d\pi_2(Y ))$$
for all $X, Y \in T_{(x,y)}(M\times N)$, where $\pi_1 : M \times N
\to M$ and $\pi_2 : M \times N \to N$ are the canonical projections.
The functions $\lambda$ and $\mu$ are called the {\it warping
functions}.

If either $\mu = 1$ or $\lambda = 1$ but not both we obtain a {\it
singly WPM}. If both $\mu = 1$ and $\lambda = 1$ then we have a {\it
direct product manifold}. If neither $\mu$ nor $\lambda$ is
constant, then we have a {\it non-trivial doubly WPM}.
\end{definition}

 We have known that preciously the formulas about
Riemann curvature and Ricci curvature are spilt into several parts
according to the horizontal lift or vertical lift of the tangent
vectors attached to the initial space $M$ or target space $N$. For
this we first introduce the unified connection and unified
Riemannian curvature on a general warped product manifold $\bar G$ (
cf. \cite{BG}, \cite{BMO}) by introducing a new notation of lift
vector.

\begin{proposition} \label{9-23-1} Let $X=(X_1,X_2),\, Y=(Y_1,Y_2)\in \mathscr{X}(\bar G)$,
where $X_1,Y_1 \in \mathscr{X}(M)$ and
 $X_2, Y_2 \in \mathscr{X}(N)$. Denote $\nabla$ by the Levi-Civita connection on the Riemannian
 product $M\times N$ with respect to the direct product metric $g=g\oplus h$ and by
 $R$ its curvature tensor field. Then the Levi-Civita connection $\bar{\nabla}$ of $\bar G$ is given by
  \begin{equation} \label{5-6-1}
  \begin{array}{ll}
  \bar{\nabla}_X Y &= \nabla_X  Y
                  + \frac{1}{2\lambda^2}X_1(\lambda^2)(0,Y_2)\\
   &               + \frac{1}{2\lambda^2}Y_1(\lambda^2)(0,X_2) -\frac{1}{2}h(X_2,Y_2)(grad\,\lambda^2,0)\\
  &=\big(\,^{M}\!\nabla_{X_1}Y_1-\frac{1}{2}h(X_2,Y_2)grad\,\lambda^2,0 \big)\\
   &+\big(0, \,^{N}\!\nabla_{X_2}Y_2+ \frac{1}{2\lambda^2}X_1(\lambda^2)Y_2
   +\frac{1}{2\lambda^2}Y_1(\lambda^2)X_2 \big),\\
    \end{array}
   \end{equation}
and the relation between the curvature tensor fields of $\bar G$ and
$M \times N$ is
  \begin{equation} \label{5-6-2}
 \begin{split}
    \bar R_{XY}-R_{XY} & = \frac{1}{2\lambda^2} \Big\{
            \Big(\,^{M}\!\nabla_{Y_1}grad_{g}\lambda^2
          - \frac{1}{2\lambda^2}Y_1(\lambda^2)grad_{g}\lambda^2,\,0 \Big)
           \wedge_{\bar g} (0, X_2)  \\
      & - \Big(\,^{M}\!\nabla_{X_1}grad_{g}\lambda^2
          - \frac{1}{2\lambda^2}X_1(\lambda^2)grad_{g}\lambda^2,\,0\Big)
           \wedge_{\bar g} (0, Y_2)\\
      & -\frac{1}{2\lambda^2}|grad_{g}\lambda^2|^2 (0,X_2)\wedge_{\bar g}(0,
      Y_2)\Big\}
 \end{split}
 \end{equation}
where the wedge product $(X\wedge_{\bar g} Y)Z =\bar g(Y,Z)X-\bar
g(X,Z)Y$, for all  $X, Y, Z \in \mathscr{X}(\bar G)$.
\end{proposition}

The detail proof see \cite{Lu3} From (\ref{5-6-2}), we easily obtain

 \begin{proposition}\label{6-30-1}
 \begin{equation} \label{5-6-3}
 \begin{split}
&\bar{R} _{(X_1, X_2)(Y_1,Y_2)} (Z_1,Z_2)\\
    &  =(\,^{M}\!R_{ X_1 Y_1} Z_1,\,^{N}\!R_{ X_2 Y_2} Z_2)\\
    & +\frac{1}{2} h(X_2,Z_2) \big( \,^{M}\!\nabla_{Y_1} \mathrm{grad}\ \lambda^2
       -\frac{1}{2\lambda^2}  Y_1(\lambda^2)\mathrm{grad}\ \lambda^2,0 \big)\\
  & -\frac{1}{2} h(Y_2,Z_2) \big( \,^{M}\!\nabla_{X_1} \mathrm{grad}\ \lambda^2
             -\frac{1}{2\lambda^2} X_1(\lambda^2)\mathrm{grad}\ \lambda^2 ,0 \big)\\
  &  +\Big(0,  \frac{1}{2\lambda^2} g ( \,^{M}\!\nabla_{X_1} \mathrm{grad}\ \lambda^2
            -\frac{1}{2\lambda^2}X_{1}(\lambda^2)\mathrm{grad}\ \lambda^2,\,Z_1 )Y_2 \big) \\
  &  -\Big(0,  \frac{1}{2\lambda^2} g ( \,^{M}\!\nabla_{Y_1} \mathrm{grad}\ \lambda^2
             -\frac{1}{2\lambda^2}Y_1(\lambda^2)\mathrm{grad}\ \lambda^2,\,Z_1)X_2 \big) \\
 & +\big(0,  \frac{1}{4\lambda^2} \mid\mathrm{grad}\ \lambda^2 \mid^2 h(X_2,Z_2)Y_2\Big)\\
  &  -\big(0,\frac{1}{4\lambda^2} \mid\mathrm{grad}\ \lambda^2 \mid^2 h(Y_2,Z_2)X_2\Big).
 \end{split}
 \end{equation}
\end{proposition}

\begin{corollary}\label{6-30-2}
 \begin{equation} \label{6-30-3}
 \begin{split}
&\bar{R} _{(X_1, X_2)(Y_1,Y_2)} (Z_1,Z_2)\\
    &  =(\,^{M}\!R_{ X_1 Y_1} Z_1,\,^{N}\!R_{ X_2 Y_2} Z_2)\\
    & +\lambda h(X_2,Z_2) \big( \,^{M}\!\nabla_{Y_1} \mathrm{grad}\ \lambda, 0 \big)
   -\lambda h(Y_2,Z_2) \big( \,^{M}\!\nabla_{X_1} \mathrm{grad}\ \lambda, 0 \big)\\
  &  + \frac{1}{\lambda}\mathrm{Hess}(\lambda)(X_1,Z_1)(0,Y_2)
     -\frac{1}{\lambda}\mathrm{Hess}(\lambda)(Y_1,Z_1)(0,X_2)\\
   & +\mid\mathrm{grad}\ \lambda \mid^2 h(X_2,Z_2)(0,Y_2)
    - \mid\mathrm{grad}\ \lambda \mid^2 h(Y_2,Z_2) (0,X_2).
 \end{split}
 \end{equation}
\end{corollary}

\begin{proof}
  Note that
  \[\begin{split}
   & \,^{M}\!\nabla_{X_1} \mathrm{grad}\ \lambda^2
       -\frac{1}{2\lambda^2}  X_1(\lambda^2)\mathrm{grad}\ \lambda^2\\
  & =\,^{M}\!\nabla_{X_1} (2\lambda \mathrm{grad}\ \lambda)
       -\frac{1}{2\lambda^2} 2\lambda X_1(\lambda) 2\lambda\mathrm{grad}\ \lambda \\
  &=2\lambda\,^{M}\!\nabla_{X_1}  \mathrm{grad}\ \lambda
  \end{split}\]
and
 \[\begin{split}
   &\frac{1}{2\lambda^2} g \left( \,^{M}\!\nabla_{X_1} \mathrm{grad}\ \lambda^2
            -\frac{1}{2\lambda^2}X_{1}(\lambda^2)\mathrm{grad}\ \lambda^2,\,Z_1  \right )\\
     &=\frac{1}{\lambda}g\left( \,^{M}\!\nabla_{X_1}  \mathrm{grad}\ \lambda,\,Z_1 \right)\\
     &= \frac{1}{\lambda}\mathrm{Hess}(\lambda)(X_1,Z_1).
  \end{split}\]
 Exchanging $X_1$ with $Y_1$, we obtain
 \[ \,^{M}\!\nabla_{Y_1} \mathrm{grad}\ \lambda^2
       -\frac{1}{2\lambda^2}  Y_1(\lambda^2)\mathrm{grad}\ \lambda^2
       =2\lambda\,^{M}\!\nabla_{Y_1}  \mathrm{grad}\ \lambda,\]
 \[ \frac{1}{2\lambda^2} g \left( \,^{M}\!\nabla_{Y_1} \mathrm{grad}\ \lambda^2
            -\frac{1}{2\lambda^2}Y_{1}(\lambda^2)\mathrm{grad}\ \lambda^2,\,Z_1  \right )
 =\frac{1}{\lambda}\mathrm{Hess}(\lambda)(Y_1,Z_1). \]
Putting these facts together, (\ref{5-6-3}) can reduce to
\begin{equation} \notag
 \begin{split}
&\bar{R} _{(X_1, X_2)(Y_1,Y_2)} (Z_1,Z_2)\\
    &  =(\,^{M}\!R_{ X_1 Y_1} Z_1,\,^{N}\!R_{ X_2 Y_2} Z_2)\\
    & +\lambda h(X_2,Z_2) \big( \,^{M}\!\nabla_{Y_1} \mathrm{grad}\ \lambda, 0 \big)
   -\lambda h(Y_2,Z_2) \big( \,^{M}\!\nabla_{X_1} \mathrm{grad}\ \lambda, 0 \big)\\
  &  + \frac{1}{\lambda}\mathrm{Hess}(\lambda)(X_1,Z_1)(0,Y_2)
     -\frac{1}{\lambda}\mathrm{Hess}(\lambda)(Y_1,Z_1)(0,X_2)\\
   & +\mid\mathrm{grad}\ \lambda \mid^2 h(X_2,Z_2)(0,Y_2)
    - \mid\mathrm{grad}\ \lambda \mid^2 h(Y_2,Z_2) (0,X_2),
 \end{split}
 \end{equation}
as claimed (\ref{6-30-3}).
\end{proof}


\section{ The first variation of Bi-$f$-energy functional and properties of
bi-$f$-harmonicity with conformal dilation}

A more natural generalization of $f$-harmonic maps and bi-harmonic
maps is given by integrating the square of the norm of the
$f$-tension field, introduced recently in \cite{OND}. The authors of
\cite{OND} give the first and second variations of bi-$f$-energy
functional. But they cannot give any example about bi-$f$-harmonic
maps. Here, we have exchanged the terminology ``$f$-bi-harmonic" in
\cite{OND} for ``bi-$f$-harmonic", which the reasons see Remark
\ref{3-24-1}.

\subsection{Bi-$f$-tension field and the first variation }

A more natural generalization by combining $f$-harmonic maps and
bi-harmonic maps is given by integrating the square of the norm of
the $f$-tension field, introduced recently in [OND]. The precise
definition is follows.
\begin{definition}
 Bi-$f$-energy functional of smooth map $\phi: (M,g) \to (N,h)$ is defined by
    \begin{equation} \label{5-17-1a}
  E_{f,2} (\phi)=\frac{1}{2} \int_{\Omega}  |\tau_f(\phi)|^2 dv_g
\end{equation}
for every compact domain $\Omega \subset M$. A map $\phi$ is called
bi-$f$-harmonic map if it the critical point of bi-$f$-energy
functional.
\end{definition}

\begin{remark} \label{3-24-1} Here we have changed the terminology ``$f-$bi-energy functional"
 and ``f-bi-harmonic map" in \cite{OND} into ``bi-$f-$energy functional" and ``bi-$f$-harmonic
 map". Since we think it is suitable for the constructing models of
 bi-harmonic map and $f-$harmonic map. $f-$Bi-energy functional " is weighted too heavily toward
 ``$f-$energy functional $E_f$" with factor $f$ in the integrand while bi-$f-$energy functional "
  toward ``$bi-$energy functional" without factor $f$. Based on above views, $f$-bi-energy functional
 seems to be the form
          \begin{equation} \label{3-24-2} \frac{1}{2} \int_{\Omega}  f|\tau(\phi)|^2
          dv_g \end{equation}
not the form (\ref{5-17-1a}).
\end{remark}

 Now we are to briefly derive the Euler-Lagrange equation gives the bi-$f$-harmonic map equation
  by using the first variation (c.f.\cite{OND})

\begin{proposition} \label{3-24-3} Let $\phi: (M,g) \to (N,h)$ be a
smooth map. Then bi-$f$-tension field of $\phi$ is
\begin{equation} \label{3-24-4}
  \tau_{f,2} (\phi)=\Delta_f^2 \tau_f(\phi)-f \rm{Tr}_g R^N(\tau_f(\phi),d\phi)d\phi=0
\end{equation}
where
   $$\Delta_f^2 \tau_f(\phi)=-\rm{Tr}_g \big( \nabla^\phi f\nabla^\phi \tau_f(\phi)
   -f \nabla^\phi_{{}^M\!\nabla_. .} \tau_f(\phi) \big).$$
For an orthonormal frame $\{e_i\}_{i=1}^m$, we have
  \begin{equation} \label{3-24-5}
  \begin{array}{ll}
  \rm{Tr}_g \big( \nabla^\phi f\nabla^\phi \tau_f(\phi) & -f \nabla^\phi_{\nabla_. .} \tau_f(\phi) \big)
      = \sum\limits_{i=1}^{m} \big( \nabla^\phi_{e_i} f\nabla^\phi_{e_i} \tau_f(\phi)
            -f \nabla^\phi_{{}^M\!\nabla_{e_i} e_i} \tau_f(\phi) \big)\\
    & =\sum\limits_{i=1}^{m} \big( f \nabla^\phi_{e_i} \nabla^\phi_{e_i} \tau_f(\phi)
    -f \nabla^\phi_{\nabla_{e_i} e_i} \tau_f(\phi)
                 +\nabla^\phi_{grad\,f} \tau_f(\phi)\big).
  \end{array}
  \end{equation}
\end{proposition}

Here we only give the outline proof of Proposition \ref{3-24-3}, for
detailed, refer to Proposition 6 in \cite{OND}. We mainly stress on
the key points in the proof.

\begin{proof} Let $\{\phi_t\}_{t\in I}$ is a smooth variation of $\phi$ and
$I=(-\varepsilon,\varepsilon)$ for some small sufficiently positive
number $\varepsilon$. Denote $\Phi(t,p)=\phi_t(p)$ and the variation
vector field $V \in \Gamma(\Phi^{-1} TN)$ associated to
$\{\phi_t\}_{t\in I}$ by
 \[ V_p=\frac{d}{dt}\Big|_{t=0} \phi_t(p)=d\Phi_{(0,p)}(\frac{\partial}{\partial t}), \qquad \forall p\in M.\]
  Taking a normal orthonormal frame  $\{e_i\}_{1\leq i \leq m}$ at $p$, we have
   \begin{equation}\label{3-24-6}
   \begin{split}
      \frac{d}{dt} E_{f,2}(\phi_t)\Big|_{t=0}&
      =\int_M \left<\nabla_{\frac{\partial}{\partial t} }^\Phi\tau_f(\phi_t)
           \Big|_{t=0},\tau_f(\phi) \right> d v_g\\
       & =\int_M \left<\sum\limits_{i}\left(\nabla_{\frac{\partial}{\partial t}}^\Phi
        \nabla_{e_i}^\Phi f d\Phi (e_i) -\nabla_{\frac{\partial}{\partial t}}^\Phi
         f d\Phi (\nabla_{e_i} e_i) \right)\Big|_{t=0},\tau_f(\phi) \right> d v_g\\
        &=\int_M  \left< \sum\limits_{i} \nabla_{e_i}^\phi f\nabla_{e_i}^\phi V_p
      + f\sum\limits_{i}R^N\big(V_p,d\phi( e_i)\big)d\phi(e_i),\tau_f(\phi) \right> dv_g
   \end{split}
   \end{equation}
   Now we manage to isolate  $V_p$ from $<R^N\big(V_p,d\phi(e_i)\big)d\phi(e_i),\tau_f(\phi)>$ and
   $<\rm{Tr}_g(\nabla^\phi f\nabla^\phi V_p),\tau_f(\phi)>$. On one hand, by the symmetric properties of
Riemann-Christoffel tensor field, we have
 \[ <f\sum\limits_{i}R^N\big(V,d\phi( e_i)\big)d\phi(e_i),\tau_f(\phi)>
    =<f\rm{Tr}_g R^N\big(\tau_f(\phi),d\phi\big)d\phi,V > .
   \]
On the other hand,  noting that the following fact
\[\big<\rm{Tr}_g(\nabla^\phi f\nabla^\phi V),\tau_f(\phi)\big>
   =\big<\rm{Tr}_g(\nabla^\phi f\nabla^\phi \tau_f(\phi)),V\big> -d^\phi<f\nabla \tau_f(\phi),V>
   +d^\phi <{\tau_f(\phi),f\nabla V}>.\]
by using the Divergence Theorem,  (\ref{3-24-6}) yields
       \begin{equation}\label{3-24-7}
   \begin{split}
     & \frac{d}{dt} E_{f,2}(\phi_t)\Big|_{t=0}\\
      & =\int_M  \big<\rm{Tr}_g(\nabla^\phi f\nabla^\phi
      \tau_f(\phi)),V\big> dv_g -\int_M d^\phi<f\nabla \tau_f(\phi),V> dv_g
     \\ & \qquad +\int_M d^\phi <{\tau_f(\phi),f\nabla V}> dv_g
      +\int_M \Big<f\rm{Tr}_g R^N\big(\tau_f(\phi),d\phi\big)d\phi,V\Big>  d v_g\\
     & =\int_M  \big<\rm{Tr}_g(\nabla^\phi f\nabla^\phi \tau_f(\phi)),V\big> dv_g
     -\int_{\partial M} <f\nabla_\nu \tau_f(\phi),V> dv_{i^\ast g}\\
    & \qquad  +\int_{\partial M} <{\tau_f(\phi),f\nabla_v V}> dv_{i^\ast g}
       +\int_M \Big<f\rm{Tr}_g R^N\big(\tau_f(\phi),d\phi\big)d\phi,V\Big>  d v_g\\
      & =\int_M  \big<\rm{Tr}_g(\nabla^\phi f\nabla^\phi
      -f\nabla_{ {}^M\nabla_{.}\cdot }) \tau_f(\phi))+f \rm{Tr}_g R^N\big(\tau_f(\phi),
           d\phi\big)d\phi,V\big>  d v_g\\
   \end{split}
   \end{equation}
where $\nu$ is the outward unit normal vector of $\partial M$ in $M$
and $i:\partial M \to M$ is the canonical inclusion. Hence, we
obtain
   \begin{equation}\label{3-24-8}
   \begin{split}
      \tau_{f,2}(\phi)
     &=-\rm{Tr}_g(\nabla^\phi f\nabla^\phi
      -f\nabla_{ {}^M\nabla_{.}\cdot }) \tau_f(\phi))-f \rm{Tr}_g R^N\big(\tau_f(\phi)\\
     &=\Delta_f^2 \tau_f(\phi))-f \rm{Tr}_g R^N\big(\tau_f(\phi),d\phi\big)d\phi,
\end{split}
   \end{equation}
as claimed.
\end{proof}

From (\ref{3-24-4} and (\ref{3-24-5}), $\tau_{f,2}(\phi)$ can
simplified as the expression
    \begin{equation}\label{3-24-08}
   \begin{split}
      \tau_{f,2}(\phi)
     =-f\rm{Tr}_g(\nabla^\phi)^2 \tau_f(\phi))-f \rm{Tr}_g
     R^N\big(\tau_f(\phi)-\nabla^\phi_{grad\,f} \tau_f(\phi)\big).
     \end{split}
   \end{equation}

\begin{remark} From (\ref{3-24-08}), we can easily see that $f-$harmonic map must be bi-$f$-harmonic map.
Conversely, it is not true. From the right hand side in
$\ref{3-24-08}$, we argue that there exist at least a non-zero
$f-$tension field $\tau_f(\phi)$ such that $\tau_{f,2}(\phi)=0$.
\end{remark}

\subsection{Properties of bi-$f$-harmonic maps with dilation}

In order to increase some sense of bi-$f$-harmonic map, we character
some properties on conformal map between equi-dimensional manifolds.

\begin{proposition}(\cite{OND}) \label{3-31-0} Let $\phi: (M^n,g) \to (N^n,h)$ be a
conformal maps with dilation $\lambda$, i.e., $\phi^\ast h=\lambda^2
g$. Then $phi$ is a bi-$f$-harmonic map if and only if
  \begin{equation} \label{3-31-1}
  \begin{array}{ll}
  &0= (n-2)f^2 d\phi\big( \rm{grad}_g(\Delta \log \lambda)\big)
     -(n-2)^2f^2\nabla_{\rm{grad}_g \log\lambda} d\phi(\rm{grad}_g \log \lambda)\\
  & \, +4(n-2)f\nabla_{\rm{grad}_g f} d\phi(\rm{grad}_g \log \lambda)
         +(n-2)f d\phi(\rm{grad}_g\lambda) \Delta f\\
 &\,- fd\phi\big(\rm{grad}_g (\Delta f)\big)+2(n-2)f^2 \left<\nabla d\phi, \nabla d\log \lambda\right>
\\ & \, -2\left<\nabla d\phi, \nabla df\right>+(n-2)|\rm{grad}_g f|^2d\phi(\rm{grad}_g\log \lambda)
  \\&\,    +2(n-2)f^2 d\phi\big({}^M\!\rm{Ric}(\rm{grad}_g\log \lambda)\big)
   \\&\,     -\nabla_{\rm{grad}_g f} d\phi(\rm{grad}_g f)
        -fd\phi \big({}^M\!\rm{Ric}(\rm{grad}_g \lambda)\big),
    \end{array}
  \end{equation}
where for a local orthonormal frame $\{e_i\}_{i=1,\ldots,n}$ on $M$,
   ${}^M\!\rm{Ric}(X)=\sum\limits_{i=1}^n R^M(X,e_i)e_i $,
  \begin{equation} \label{5-14-1}
  \begin{split}
  \left<\nabla d\phi, \nabla d\log \lambda\right>
  & =\sum\limits_{i,j=1}^n\nabla d\phi(e_i,e_j)\nabla d\log \lambda(e_i,e_j)\\
  &=\sum\limits_{i,j=1}^n\nabla d\phi(e_i,e_j) g({}^M\!\nabla_{e_i} \rm{grad}_g\log \lambda, e_j)
  \\ &=\sum\limits_{i=1}^n\nabla d\phi(e_i,{}^M\!\nabla_{e_i} \rm{grad}_g\log \lambda).
  \end{split}
  \end{equation}
similar to $\left<\nabla d\phi, \nabla df\right>$.
\end{proposition}

\begin{proof} By the assumption of $\phi$ and (\ref{1.1}), the
$f-$tension field of $\phi$ is given by
  \begin{equation} \label{3-31-2a}
  \tau_f(\phi)=(2-n)fd\phi(\rm{grad}_g \log \lambda)+d\phi(\rm{grad}_g f).
  \end{equation}
Using (\ref{3-24-4}), a long deduction (for detail, see \cite{OND},
pp.22-24 ) gives
      \begin{equation} \label{3-31-2}
  \begin{array}{ll}
  &\tau_{f,2}(\phi)= (n-2)f^2 d\phi\big( \rm{grad}_g(\Delta \log \lambda)\big)
     -(n-2)^2f^2\nabla_{\rm{grad}_g \log\lambda } d\phi(\rm{grad}_g \log \lambda)\\
  & \, +4(n-2)f\nabla_{\rm{grad}_g f} d\phi(\rm{grad}_g \log \lambda)
         +(n-2)f d\phi(\rm{grad}_g\lambda) \Delta f\\
 &\,- fd\phi\big(\rm{grad}_g (\Delta f)\big)+2(n-2)f^2 \left<\nabla d\phi, \nabla d\log \lambda\right>
\\ & \, -2\left<\nabla d\phi, \nabla df\right>+(n-2)|\rm{grad}_g f|_g^2d\phi(\rm{grad}_g\log \lambda)
  \\&\,    +2(n-2)f^2 d\phi\big({}^M\!\rm{Ric}(\rm{grad}_g\log \lambda)\big)
   \\&\,     -\nabla_{\rm{grad}_g f} d\phi(\rm{grad}_g f)
        -fd\phi \big({}^M\!\rm{Ric}(\rm{grad}_g \lambda)\big).
    \end{array}
  \end{equation}
Thus the necessary and sufficient condition for bi-$f$-harmonic map
of $\phi$ is clear.
\end{proof}

In particular, if we consider identity map $\phi=Id_M$, then from
$\lambda=1$ we have

\begin{corollary}(\cite{OND}) \label{3-31-3a} \quad Identity map $Id_M:
(M^n,g) \to (M^n,g)$ is a bi-$f$-harmonic map if and only if
  \begin{equation} \notag
  \begin{array}{ll}
& 2d\phi\big({}^M\!\rm{Ric}(\rm{grad}_g\log \lambda)\big)
      +\rm{grad}_g(\Delta \log \lambda)
  \\ & \quad +\frac{3}{2}\rm{grad}_g (|\rm{grad}_g f|_g^2)+2|\rm{grad}_g f|_g^2\rm{grad}_g\log f=0.
  \end{array}
  \end{equation}
  \end{corollary}

Observe that if $f=\lambda$, then (\ref{3-31-2a}) is of the form
      \[ \tau_f(\phi)=(3-n)fd\phi(\rm{grad}_g \log \lambda). \]
Hence, when $n\geq 4$, we obtain

\begin{corollary}(\cite{OND}) \label{3-31-4}\quad Let $\phi:
(M^n,g) \to (N^n,g)$\quad $(n\geq 4)$ be a conformal map with
dilation $\lambda =f$. Then $\phi$ is a bi-$\lambda$-harmonic map if
and only if
  \begin{equation} \label{3-31-5}
  \begin{array}{ll}
 &d\phi\big({}^M\!\rm{Ric}(\rm{grad}_g\log\lambda)\big)
     +\rm{grad}_g( \log \lambda)-(\Delta \log\lambda)\rm{grad}_g \log\lambda
  \\ & \quad +\frac{9-n}{2}\rm{grad}_g (|\rm{grad}_g f|_g^2)+(7-n)|\rm{grad}_g f|_g^2\rm{grad}_g\log f=0.
  \end{array}
  \end{equation}
  \end{corollary}

\subsection{Examples}

 Following the method of Examples A and B in \cite[P107]{CET}, and
noting that the function $ f: M \times N \to (0,+\infty)$  and $
f_\phi: M \to (0,+\infty) $ which differs from our discussing
function $f: M \to N$ in this paper, we give two relatively simple
examples.

\begin{example} \label{5-13-1} Let $\phi: \mathbb{R} \to (N^2,h)$ be a conformal map with
constant dilation $\lambda$. If
     $f_\phi=f_N\circ \phi=\gamma: \mathbb{R}^2  \to \mathbb{R}$
be a smooth function, where, at
   $ (x,y) \in \mathbb{R}^2$, $f_N: N \to (0,+\infty)$
defined by  $f_N(y) = f(x,y)$ for all $y\in N$, then $\phi$ is
bi-$f$-harmonic if and only if

     \begin{equation} \label{5-13-2}
     \left\{  \begin{array}{l}
    \frac{ \partial \gamma}{\partial x} \frac{ \partial^2 \gamma}{\partial x^2}
        + \frac{ \partial \gamma}{\partial y} \frac{ \partial^2 \gamma}{\partial x \partial y}
        =0,\\
      \frac{ \partial \gamma}{\partial y} \frac{ \partial^2 \gamma}{\partial y^2}
        + \frac{ \partial \gamma}{\partial x} \frac{ \partial^2 \gamma}{\partial x \partial y}
        =0.
    \end{array} \right.
    \end{equation}
\end{example}

\begin{example} \label{5-13-3}
Let $\phi: (M^2,g) \to (N^2,h)$ be a conformal map with constant
dilation $\lambda$. If $f_\phi=f_N\circ \phi=\log \lambda$ be a
smooth function, then $\phi$ is bi-$f$-harmonic if and only if
$\lambda$ satisfies
     \begin{equation} \label{5-13-4}
     \rm{grad}_g (\big| \rm{grad}_g \log\lambda \big|^2)=0.
    \end{equation}
\end{example}


\section{Definition and conformal properties of $f$-Bi-harmonic maps }

\subsection{The first variation of $f$-bi-energy functional}

Another more natural generalization of combining $f$-harmonic maps
and bi-harmonic maps is given by integrating the square of the norm
of the tension field times $f$. The precise definition is follows.
\begin{definition}
 $f$-Bi-energy functional of smooth map $\phi: (M,g) \to (N,h)$ is defined by
    \begin{equation} \label{5-11-1}
  E_{2,f} (\phi)=\frac{1}{2} \int_{\Omega}  f|\tau(\phi)|^2 dV
\end{equation}
for every compact domain $\Omega \subset M$. A map $\phi$ is called
$f$-bi-harmonic map if it the critical point of $f$-bi-energy
functional.
\end{definition}

 Now we are to derive the Euler-Lagrange equation gives the $f$-bi-harmonic map equation
  by using the first variation.

\begin{proposition} \label{5-11-3} Let $\phi: (M,g) \to (N,h)$ be a
smooth map. Then $f$-bi-tension field of $\phi$ is
\begin{equation} \label{5-11-4}
    \begin{split}
    \tau_{2,f}(\phi)&=-\rm{Tr}_g(\nabla^\phi)^2 f \tau(\phi)
      -f\rm{Tr}_g R^N\big(\tau(\phi),d\phi\big)d\phi
         \\ & =-J^\phi(f \tau(\phi) \,),
      \end{split}
\end{equation}
where $J^\phi$ is the Jacobi operator along the map $\phi$.
\end{proposition}

\begin{proof} Let $\Phi: I\times M \to M$ be a smooth map satisfying
      \[   \Phi(t,p)=\phi_t(p),\qquad \Phi(0,p)=\phi(p),\qquad \forall
                        t\in I, p\in M,\]
where $\{\phi_t\}_{t\in I}$ is a smooth variation of $\phi$ and
$I=(-\varepsilon,\varepsilon)$ for some small sufficiently positive
number $\varepsilon$.

The variation vector field $V \in \Gamma(\Phi^{-1} TN)$ associated
to $\{\phi_t\}_{t\in I}$ is given by
 \[ V_p=\frac{d}{dt}\Big|_{t=0} \phi_t(p)=d\Phi_{(0,p)}(\frac{\partial}{\partial t}), \qquad \forall p\in M.\]
We have
   \begin{equation}\label{5-11-6}
   \begin{split}
      \frac{d}{dt} E_{2,f}(\phi_t)\Big|_{t=0}& =\frac{1}{2} \int_M \frac{\partial}{\partial t}
            f \left<\tau(\phi_t),\tau(\phi_t)\right >\Big|_{t=0} dv_g\\
    &=\int_M f \left<\nabla_{\frac{\partial}{\partial t} }^\Phi\tau(\phi_t),\tau(\phi_t) \right>
        \Big|_{t=0}d v_g\\
    &=\int_M f \left<\nabla_{\frac{\partial}{\partial t} }^\Phi\tau (\phi_t) \Big|_{t=0},\tau(\phi) \right>
       d v_g.
   \end{split}
   \end{equation}

Now let $\{e_i\}_{1\leq i \leq m}$ be a normal orthonormal frame at
$p\in M$. Then calculating at $p$ gives
    \begin{equation}\label{5-11-7}
   \begin{split}
      \nabla_{\frac{\partial}{\partial t}}^\Phi \tau(\phi_t)
           &=\nabla_{\frac{\partial}{\partial t}}^\Phi \rm{Tr}_g  \nabla d\Phi\\
           &= \nabla_{\frac{\partial}{\partial t}}^\Phi \sum\limits_{i} \nabla  d\Phi (e_i,e_i) \\
           &= \nabla_{\frac{\partial}{\partial t}}^\Phi \sum\limits_{i} (\nabla_{e_i}^\Phi d\Phi) (e_i) \\
             &= \sum\limits_{i}\left(\nabla_{\frac{\partial}{\partial t}}^\Phi  \nabla_{e_i}^\Phi  d\Phi (e_i)
                    -\nabla_{\frac{\partial}{\partial t}}^\Phi   d\Phi (\nabla_{e_i} e_i) \right). \\
            \end{split}
   \end{equation}
For a given $X\in \Gamma(T M)$, note that $[\frac{\partial}{\partial
t},X]=0$, we get
       \begin{equation}\notag
   \begin{split}
      \nabla_{\frac{\partial}{\partial t}}^\Phi  d\Phi (X)
   & =  \nabla_X^\Phi d\Phi(\frac{\partial}{\partial t})
        +  d\Phi \left([\frac{\partial}{\partial t},X] \right)\\
      &= \nabla_{X}^\Phi d\Phi(\frac{\partial}{\partial t}).
            \end{split}
   \end{equation}
Thus at $p$, (\ref{5-11-7}) becomes
      \begin{equation}\label{5-11-8}
   \begin{split}
      \nabla_{\frac{\partial}{\partial t}}^\Phi \tau(\phi_t)
      &= \sum\limits_{i}\left(\nabla_{\frac{\partial}{\partial t}}^\Phi  \nabla_{e_i}^\Phi  d\Phi (e_i) -
                    \nabla_{\nabla_{e_i} e_i}^\Phi   d\Phi (\frac{\partial}{\partial t})\right) \\
      &=\sum\limits_{i}\nabla_{\frac{\partial}{\partial t}}^\Phi  \nabla_{e_i}^\Phi d\Phi(e_i)\\
      &=\sum\limits_{i}\left(  \nabla_{e_i}^\Phi \nabla_{\frac{\partial}{\partial t}}^\Phi d\Phi(e_i)+
         \nabla_{[\frac{\partial}{\partial t},e_i]}^\Phi
          d\Phi(e_i)+R^\Phi(\frac{\partial}{\partial t}, e_i) d\Phi (e_i) \right)\\
      &=\sum\limits_{i}\left(  \nabla_{e_i}^\Phi \nabla_{\frac{\partial}{\partial t}}^\Phi d\Phi(e_i)+
      R^\Phi(\frac{\partial}{\partial t}, e_i) d\Phi (e_i) \right)\\
      &=\sum\limits_{i}\left(  \nabla_{e_i}^\Phi \nabla_{e_i}^\Phi d\Phi(\frac{\partial}{\partial t})
      +  R^N\big(d\Phi(\frac{\partial}{\partial t}),d\Phi( e_i)\big))fd\Phi (e_i)\right)\\
      &=\sum\limits_{i} \nabla_{e_i}^\Phi \nabla_{e_i}^\Phi d\Phi(\frac{\partial}{\partial t})
      + \sum\limits_{i}R^N\big(d\Phi(\frac{\partial}{\partial t}),d\Phi( e_i)\big)d\Phi (e_i).
            \end{split}
            \end{equation}

      Noticing the symmetric properties of Riemann-Christoffel tensor
       field,  from (\ref{5-11-8}), we have
 \begin{equation}\label{5-11-9}
   \begin{split}
     & \frac{d}{dt} E_{2,f}(\phi_t)\Big|_{t=0}\\
      & =\int_M f \left( \big<\sum\limits_{i} \nabla_{e_i}^\phi \nabla_{e_i}^\phi V,\tau(\phi)\big>
      + \Big<\sum\limits_{i}R^N\big(V,d\phi( e_i)\big)d\phi (e_i),\tau(\phi)\Big> \right) d v_g\\
     &=\int_M f\left( \big<\rm{Tr}_g(\nabla^\phi \nabla^\phi V),\tau(\phi)\big>
      -\Big<\sum\limits_{i}R^N\big(d\phi( e_i),\tau(\phi)\big)d\phi (e_i),V\Big> \right) d v_g\\
      &=\int_M \left( \big<\rm{Tr}_g(\nabla^\phi \nabla^\phi V),f \tau(\phi)\big>
      +\Big<f\rm{Tr}_g R^N\big(\tau(\phi),d\phi\big)d\phi,V\Big> \right) d v_g,
   \end{split}
   \end{equation}
 where
    \[ \rm{Tr}_g(\nabla^\phi \nabla^\phi V)=\sum\limits_{i}\big(\nabla_{e_i}^\phi \nabla_{e_i}^\phi V
       -\nabla_{{}^M\!\nabla_{e_i}e_i}^\phi V \big)=\rm{Tr}_g(\nabla^\phi)^2 V. \]
Denote by $\nu$ the outward unit normal vector of $\partial M$ in
$M$ and by $i:\partial M \hookrightarrow M$ the canonical inclusion.
Note that
   \[ \big<\rm{Tr}_g(\nabla^\phi \nabla^\phi V), f\tau(\phi)\big>
   =\big<V, \rm{Tr}_g(\nabla^\phi \nabla^\phi f \tau(\phi))\big>
      +d^\phi\big< \nabla^\phi V,f \tau(\phi)\big>-d^\phi \big<  V, \nabla^\phi f\tau(\phi)\big>,\]
   where
   \[d^\phi\big< \nabla^\phi V,f \tau(\phi)\big>=\big< \rm{Tr}_g (\nabla^\phi \nabla^\phi V), f \tau(\phi)\big>
   +\sum\limits_{i} \big< \nabla_{e_i}^\phi V, \nabla^\phi_{e_i}f\tau(\phi) \big>,\]
   \[d^\phi\big< V, \nabla^\phi f \tau(\phi)\big>=\Big< V, \rm{Tr}_g \big(\nabla^\phi \nabla^\phi f \tau(\phi)\big) \Big>
   +\sum\limits_{i} \big< \nabla_{e_i}^\phi V, \nabla^\phi_{e_i}f\tau(\phi) \big>.\]
By using the Divergence Theorem, (\ref{5-11-9}) gives
       \begin{equation}\label{5-11-10}
   \begin{split}
     & \frac{d}{dt} E_{2,f}(\phi_t)\Big|_{t=0}\\
      & =\int_M  \big<\rm{Tr}_g(\nabla^\phi f\nabla^\phi \tau(\phi)),V\big>
     +\int_M  d^\phi\big< \nabla^\phi V,f \tau(\phi)\big> dv_g
                  -\int_M d^\phi \big<  V, \nabla^\phi f\tau(\phi)\big> dv_g\\
      & \qquad  +\int_M \Big<f\rm{Tr}_g R^N\big(\tau_f(\phi),d\phi\big)d\phi,V\Big>  d v_g\\
     & =\int_M  \big<\rm{Tr}_g(\nabla^\phi \nabla^\phi f \tau(\phi))
      +f\rm{Tr}_g R^N\big(\tau(\phi),d\phi\big)d\phi,V\big> dv_g\\
    &\qquad  +\int_{\partial M} \big< \nabla^\phi_\nu V,f \tau(\phi)\big> dv_{i^\ast g}
       -\int_{\partial M} \big<  V, \nabla^\phi_\nu f\tau(\phi)\big> dv_{i^\ast g}\\
     & =\int_M  \big<\rm{Tr}_g(\nabla^\phi \nabla^\phi f \tau(\phi))
      +f\rm{Tr}_g R^N\big(\tau(\phi),d\phi\big)d\phi,V\big> dv_g.
   \end{split}
   \end{equation}
Thus, (\ref{5-11-10}) reduces to
     \begin{equation}\label{5-11-11}
   \begin{split}
      \frac{d}{dt} E_{2,f}(\phi_t)\Big|_{t=0} &  =-\int_M  \Big<-\rm{Tr}_g\big(\nabla^\phi \nabla^\phi f \tau(\phi)
      -\nabla_{ {}^M\!\nabla_\cdot \cdot }^\phi\big) f\tau(\phi)\\
    & \qquad   -f\rm{Tr}_g R^N\big(\tau(\phi),d\phi\big)d\phi,V\Big> dv_g,
   \end{split}
   \end{equation}
from which, the Euler-Lagrange operator associated with $\phi$ is
given by
    \begin{equation}\notag
    \begin{split}
     \tau_{2,f}(\phi)&=-\rm{Tr}_g(\nabla^\phi \nabla^\phi-\nabla_{ {}^M\!\nabla_{\cdot}\cdot}^\phi) f\tau(\phi)
      -f\rm{Tr}_g R^N\big(\tau(\phi),d\phi\big)d\phi\\
     &  =-\rm{Tr}_g(\nabla^\phi)^2 f\tau(\phi)
      -f\rm{Tr}_g R^N\big(\tau(\phi),d\phi\big)d\phi.
     \end{split}
     \end{equation}
\end{proof}

Next, we give the relation between $f$-bi-tension field
$\tau_{2,f}(\phi)$ and bi-tension field $\tau_2(\phi)$.

\begin{proposition} \label{5-12-1} Let $\phi: (M,g) \to (N,h)$ be a
smooth map. Then the relation between $f$-bi-tension field
$\tau_{2,f}(\phi)$ and bi-tension field $\tau_2(\phi)$ is
\begin{equation} \label{5-12-2}
\begin{split}
    \tau_{2,f}(\phi)=f \tau_2(\phi)-\Delta(f) \tau(\phi)-2\nabla^\phi _{\rm{grad}_g f} \tau(\phi).
    \end{split}
\end{equation}
where $\tau_2(\phi)=-\rm{Tr}_g(\nabla_{.}^\phi)^2
\tau(\phi)-\rm{Tr}_g R^N\big(\tau(\phi),d\phi\big)d\phi$.
\end{proposition}

\begin{proof} Under a local orthonormal basis
$\{e_i\}_{i=1,\ldots,m}$ over $M$, we have
     \begin{equation}\label{5-12-3}
   \begin{split}
      &\rm{Tr}_g (\nabla^\phi)^2  f\tau(\phi)
     \\ &= \rm{Tr}_g \big(\nabla^\phi \nabla^\phi f\tau(\phi)
            -\nabla_{ {}^M\!\nabla_{\cdot}\cdot}^\phi f\tau(\phi)\big)
     \\ &   =\sum\limits_{i}\big(\nabla_{e_i}^\phi \nabla_{e_i}^\phi
        f\tau(\phi)  -\nabla_{{}^M\!\nabla_{e_i}e_i}^\phi f\tau(\phi) \big)\\
      & =f  \sum\limits_{i}\big(\nabla_{e_i}^\phi \nabla_{e_i}^\phi
        \tau(\phi) -\nabla_{{}^M\!\nabla_{e_i}e_i}^\phi \tau(\phi)\big) \\
      &\qquad  +\sum\limits_{i}\big( e_i(e_i f)-{}^M\!\nabla_{e_i} e_i (f)\big) \tau(\phi)
              + 2\sum\limits_{i}e_i (f)\nabla^\phi_{e_i} \tau(\phi)
        \\ & = f\rm{Tr}_g (\nabla^\phi
        \nabla^\phi -\nabla_{ {}^M\!\nabla_{\cdot}\cdot}^\phi)\tau(\phi)
        +\Delta_M(f)\tau(\phi)+2 \nabla_{\rm{grad}_g f}^\phi \tau(\phi).
   \end{split}
   \end{equation}
Combining (\ref{1.1}) and (\ref{5-11-4}), we obtain
  \[ \tau_{2,f}(\phi)=f \tau_2(\phi)-\Delta(f) \tau(\phi)- 2\nabla^\phi _{\rm{grad}_g f}
  \tau(\phi),\]
as claimed.
\end{proof}

\begin{remark} \label{5-12-3a} In fact, the proof of Proposition (\ref{5-12-1}) can also follow
from Equation (7) in \cite{Ou2}.
\end{remark}
  \begin{definition}
$f$-Bi-energy functional of smooth map $\phi: (M,g) \to (N,h)$ is
defined by
    \begin{equation} \label{5-11-1}
  E_{2,f} (\phi)=\frac{1}{2} \int_{\Omega}  f|\tau(\phi)|^2 dV
\end{equation}
for every compact domain $\Omega \subset M$. A map $\phi$ is called
$f$-bi-harmonic map if it the critical point of $f$-bi-energy
functional.
\end{definition}

\begin{remark} \label{5-18-2a} From (\ref{5-12-2}), we can easily see that
 the notion of $f$-bi-harmonic maps generalizes the notion of harmonic maps and
 that of bi-harmonic maps because a harmonic map is always an $f$-bi-harmonic map
 for any function $f>0$. Also, an $f$-bi-harmonic map with constant $f$ is nothing
 but a bi-harmonic map.
\end{remark}
\begin{remark} \label{5-18-2b}
Proposition \ref{5-12-1} provides a simple path to find a
non-trivial $f$-bi-harmonic map. As long as let $\tau_2(\phi)=0$ and
test whether there exist a suitable map $\phi$ and a non-constant
positive function $f$ such that they are a solution to
   \[ \Delta_M(f)\tau(\phi)+2 \nabla_{\rm{grad}_g f}^\phi\tau(\phi)=0. \]
\end{remark}

\subsection{ $f$-Bi-harmonic maps with conformal dilation}

In order to increase some sense of $f$-bi-harmonic map, we character
some properties on conformal map between equi-dimensional manifolds
by borrowing from the idea in \cite{BFO, OND, CET}.

\begin{proposition} \label{5-14-01} Let $\phi: (M^n,g) \to (N^n,h)$ be a
conformal maps with dilation $\lambda$, i.e., $\phi^\ast h=\lambda^2
g$. Then under an assumption that $n\geq 3$,  $phi$ is a
$f$-bi-harmonic map if and only if
  \begin{equation} \label{5-14-02}
  \begin{array}{ll}
  & (2-n)f \nabla_{\rm{grad}_g \log\lambda} d\phi(\rm{grad}_g \log \lambda)
      +f d\phi\big( \rm{grad}_g(\Delta \log \lambda)\big)
  \\ &  +2d\phi\big({}^M\!\rm{Ric}(\rm{grad}_g\log \lambda)
     +2\left<\nabla d\phi, \nabla d\log \lambda\right>
   \\&    +\Delta(f) d\phi(\rm{grad}_g \log \lambda)
   +2\nabla^\phi _{\rm{grad}_g f} d\phi(\rm{grad}_g \log \lambda)=0,
    \end{array}
  \end{equation}
where $\left<\nabla d\phi, \nabla d\log \lambda\right>$ is defined
by (\ref{5-14-1}).
\end{proposition}

\begin{proof} Note that the fundamental equation for the
tension field of horizontally conformal submersion to the mean
curvature $\mu^\nu$ of its fibres and the horizontal gradient of its
dilation $\lambda$ is given by
      \begin{equation} \label{5-14-2}
    \tau(\phi)=(2-n)d\phi(\rm{grad}_g \log
    \lambda)-(m-n)d\phi(\mu^\nu),
   \end{equation}
 (see Proposition 4.5.3 in \cite{BW}).  By the assumption of equiv-dimension, i.e., $m=n$, the
tension field of $\phi$ is given by
  \begin{equation} \label{5-14-3}
  \tau(\phi)=(2-n)d\phi(\rm{grad}_g \log \lambda).
  \end{equation}

Using (\ref{5-12-2}) and (\ref{1.1}), we have
      \begin{equation} \label{5-14-4}
  \begin{array}{ll}
  &\tau_{2,f}(\phi)= (n-2)f \rm{Tr}_g(\nabla^\phi)^2 d\phi(\rm{grad}_g \log \lambda)\\
    &\quad  +(n-2)f  \rm{Tr}_g R^N\big( d\phi(\rm{grad}_g \log \lambda), d\phi)d\phi\\
     &+(n-2)\Delta(f) d\phi(\rm{grad}_g \log \lambda)
     +2(n-2)\nabla^\phi _{\rm{grad}_g f} d\phi(\rm{grad}_g \log \lambda).
    \end{array}
  \end{equation}

By Lemma B in \cite{CET} or Corollary 2.2 in \cite{BFO}, we have
  \begin{equation} \label{5-14-5}
   \begin{array}{ll}
  &\rm{Tr}_g(\nabla^\phi)^2 d\phi(\rm{grad}_g \log \lambda)
    + \rm{Tr}_g R^N\big( d\phi(\rm{grad}_g \log \lambda), d\phi)d\phi
  \\ &= (2-n)\nabla_{\rm{grad}_g \log\lambda} d\phi(\rm{grad}_g \log \lambda)+
   d\phi\big( \rm{grad}_g(\Delta \log \lambda)\big)
   \\  & \quad +2 d\phi\big({}^M\!\rm{Ric}(\rm{grad}_g\log \lambda)
         +2\left<\nabla d\phi, \nabla d\log \lambda\right>.
    \end{array}
  \end{equation}
Substituting (\ref{5-14-5}) into (\ref{5-14-4}), we obtain
   \begin{equation} \label{5-14-6}
  \begin{array}{ll}
  &\tau_{2,f}(\phi)= -(n-2)^2f \nabla_{\rm{grad}_g \log\lambda} d\phi(\rm{grad}_g \log \lambda)\\
    &\quad  +(n-2)f d\phi\big( \rm{grad}_g(\Delta \log \lambda)\big)
    +2(n-2)d\phi\big({}^M\!\rm{Ric}(\rm{grad}_g\log \lambda)\\
     &+2(n-2)\left<\nabla d\phi, \nabla d\log \lambda\right>
     +(n-2)\Delta(f) d\phi(\rm{grad}_g \log \lambda)
    \\& + 2(n-2)\nabla^\phi _{\rm{grad}_g f} d\phi(\rm{grad}_g \log \lambda),
    \end{array}
  \end{equation}
from which, $\phi$ is a $f$-bi-harmonic map ($n>2$) if and only if
     \begin{equation} \notag
  \begin{array}{ll}
  & (2-n)f \nabla_{\rm{grad}_g \log\lambda} d\phi(\rm{grad}_g \log \lambda)
      +f d\phi\big( \rm{grad}_g(\Delta \log \lambda)\big)
 \\&\quad    +2d\phi\big({}^M\!\rm{Ric}(\rm{grad}_g\log \lambda)
     +2\left<\nabla d\phi, \nabla d\log \lambda\right>
 \\ &\quad     +\Delta(f) d\phi(\rm{grad}_g \log \lambda)
    + 2\nabla^\phi _{\rm{grad}_g f} d\phi(\rm{grad}_g \log \lambda)=0,
    \end{array}
  \end{equation}
as claimed.

\end{proof}

In particular, if we consider identity map $\phi=Id_M$, then from
$\lambda=1$ we have

\begin{corollary} \label{3-31-3} \quad Identity map $Id_M:
(M^n,g) \to (M^n,g)$ is a $f$-bi-harmonic map
  \end{corollary}

  \begin{remark} \label{5-18-2aa} By Remark \ref{5-18-2aa},
  Corollary \ref{3-31-3} indeed is a trivial conclusion. More generally,
  any isometry is harmonic and hence an $f$-bi-harmonic map for any $f$.
  \end{remark}

Observe that (\ref{5-14-3}) contains such two terms
  $\nabla^\phi_{\rm{grad}_g f} d\phi(\rm{grad}_g \log \lambda)$
 \\ and  $f \nabla_{\rm{grad}_g \log\lambda} d\phi(\rm{grad}_g  \log\lambda)$,
  when $f=\lambda$, we have

\begin{proposition} \label{5-15-1}\quad Let $\phi:
(M^n,g) \to (N^n,g)$\quad $(n\geq 3)$ be a conformal map with
dilation $\lambda =f$. Then $\phi$ is a $f$-bi-harmonic map if and
only if
   \begin{equation} \label{5-15-2}
  \begin{split}
 & \Big((5-n)\lambda |\rm{grad}_g \log \lambda|^2
             +(\lambda-2)(\Delta \log \lambda) \Big) \rm{grad}_g \log \lambda
 \\ &  +\Big( (4-n)\lambda +2)\Big) {}^M\!\nabla_{\rm{grad}_g \log\lambda} \rm{grad}_g \log \lambda
         + \rm{grad}_g \big( | \rm{grad}_g \log \lambda |^2\big)
 \\ & \quad  + \lambda \rm{grad}_g(\Delta \log \lambda)
              + 2 {}^M\!\rm{Ric}(\rm{grad}_g\log \lambda)=0,
    \end{split}
  \end{equation}
  \end{proposition}

\begin{proof} From (\ref{5-14-3}), it is clear that if $n=2$, then a conformal
map $\phi$ is harmonic so $\lambda$-bi-harmonic.

 Now, we consider $n\geq 3$ and calculate the $\lambda$-bi-tension field. By (\ref{5-14-6}),
 we have
   \begin{equation} \label{5-15-3}
  \begin{array}{ll}
  &\tau_{2,\lambda}(\phi)= (n-2)(4-n)\lambda \nabla_{\rm{grad}_g \log\lambda} d\phi(\rm{grad}_g \log \lambda)\\
    &\quad  +(n-2)\lambda d\phi\big( \rm{grad}_g(\Delta \log \lambda)\big)
    +2(n-2)d\phi\big({}^M\!\rm{Ric}(\rm{grad}_g\log \lambda)\\
     &+2(n-2)\left<\nabla d\phi, \nabla d\log \lambda\right>
     +(n-2)\Delta(\lambda) d\phi(\rm{grad}_g \log \lambda).
    \end{array}
  \end{equation}
In order to use the relation $\phi^\ast h=\lambda^2 g$, we consider
the equivalence of $\tau_{2,\lambda}=0$ to
     \[  h(\tau_{2,\lambda}(\phi), d\phi)=0.\]
From (\ref{5-15-3}), we conclude that $\phi$ is
$\lambda$-bi-harmonic if and only if for any $X\in \Gamma(TM)$, we
have
    \begin{equation} \label{5-15-4}
  \begin{array}{ll}
 &(4-n)\lambda h\big(\nabla_{\rm{grad}_g \log\lambda} d\phi(\rm{grad}_g \log \lambda),d\phi(X) \big)\\
    &\quad  +\lambda h\Big(d\phi\big( \rm{grad}_g(\Delta \log \lambda)\big), d\phi(X) \Big)
    +2 h\big( d\phi\big({}^M\!\rm{Ric}(\rm{grad}_g\log \lambda), d\phi(X) \big)\\
     &+2 h\big(\left<\nabla d\phi, \nabla d\log \lambda\right>,  d\phi(X) \big)
     + h \big(\Delta(\lambda) d\phi(\rm{grad}_g \log \lambda), d\phi(X) \big)  =0.
    \end{array}
  \end{equation}
We will study term by term of (\ref{5-15-4}).

On one hand, note that
 Equation (i) of Lemma 4.5.1 in \cite{BW} is stated by
    \begin{equation} \label{5-15-5}
  \begin{split}
 \nabla d\phi(X,Y)&=X(\log \lambda)
 d\phi(Y)+Y(\log\lambda)d\phi(X)-g(X,Y)d\phi(\rm{grad}_g \log\lambda)
 \\ &=d\phi\big(X(\log \lambda)Y+Y(\log \lambda)X-g(X,Y)\rm{grad}_g \log\lambda
 \big), \quad \forall X,Y\in \Gamma(TM),
  \end{split}
  \end{equation}
we have
   \begin{equation} \label{5-15-6}
  \begin{split}
 & \nabla_{\rm{grad}_g \log\lambda} d\phi(\rm{grad}_g \log \lambda)
\\ &=\nabla d\phi(\rm{grad}_g \log \lambda, \rm{grad}_g
 \log\lambda)+d\phi( {}^M\!\nabla_{\rm{grad}_g \log\lambda} \rm{grad}_g \log \lambda )
  \\ &=d\phi\big(|\rm{grad}_g \log \lambda|^2 \rm{grad}_g \log \lambda\big)
         +d\phi({}^M\!\nabla_{\rm{grad}_g \log\lambda} \rm{grad}_g \log \lambda \big)
  \end{split}
  \end{equation}
and
     \begin{equation} \label{5-15-7a}
  \begin{array}{ll} \left<\nabla d\phi, \nabla d\log \lambda\right>
  & =\sum\limits_{i=1}^n\nabla d\phi(e_i,\nabla_{e_i}\rm{grad}_g \log \lambda)
  \\ & =\sum\limits_{i=1}^n \big( e_i(\log\lambda ) d\phi(\nabla_{e_i}\rm{grad}_g \log \lambda )
  \\ & \quad +(\nabla_{e_i}\rm{grad}_g \log \lambda)(\log\lambda) d\phi(e_i)
  \\ & \quad - g(e_i, \nabla_{e_i}\rm{grad}_g \log \lambda) d\phi(\rm{grad}_g \log \lambda) \big)
  \\ & =d\phi({}^M\!\nabla_{\rm{grad}_g \log\lambda} \rm{grad}_g \log \lambda))
  \\ &\quad + d\phi(\frac{1}{2}\rm{grad}_g \big( | \rm{grad}_g \log \lambda |^2\big)
              -\Delta(\log\lambda)d\phi(\rm{grad}_g \log \lambda)
  \\ &= d\phi\Big({}^M\!\nabla_{\rm{grad}_g \log\lambda} \rm{grad}_g \log \lambda)
     \\ & \quad + \frac{1}{2}\rm{grad}_g \big( | \rm{grad}_g \log \lambda |^2\big)
         -\Delta(\log\lambda)\rm{grad}_g \log \lambda\Big).
\end{array}
  \end{equation}

  From (\ref{5-15-6}) and (\ref{5-15-7a}),respectively, we get
      \begin{equation} \label{5-15-7}
  \begin{split}
 & h\big(\nabla_{\rm{grad}_g \log\lambda} d\phi(\rm{grad}_g \log \lambda), d\phi(X)\big)
   \\ &=\lambda^2 g(|\rm{grad}_g \log \lambda|^2\rm{grad}_g \log \lambda
   + {}^M\!\nabla_{\rm{grad}_g \log\lambda} \rm{grad}_g \log \lambda, X\Big),
     \\ & h\big(\left<\nabla d\phi, \nabla d\log \lambda\right>,  d\phi(X) \big)
   \\ & \quad = \lambda^2 g\Big({}^M\!\nabla_{\rm{grad}_g \log\lambda} \rm{grad}_g \log \lambda)
     \\ & \quad + \frac{1}{2}\rm{grad}_g \big( | \rm{grad}_g \log \lambda |^2\big)
         -\Delta(\log\lambda)\rm{grad}_g \log \lambda, X \Big).
   \end{split}
  \end{equation}
On the other hand, we have
    \begin{equation} \label{5-15-8}
  \begin{split}
 &h\Big(d\phi\big( \rm{grad}_g(\Delta \log \lambda)\big), d\phi(X)
 \Big)=\lambda^2 g\big(\rm{grad}_g(\Delta \log \lambda),X \big ),
     \\ &
   \\ & h\big( d\phi\big({}^M\!\rm{Ric}(\rm{grad}_g\log \lambda),
          d\phi(X)=\lambda^2 g({}^M\!\rm{Ric}(\rm{grad}_g\log \lambda),X)
     \\ &
   \\ & h \big( (\Delta\lambda) d\phi(\rm{grad}_g \log \lambda), d\phi(X) \big)
    = \lambda^2 g\big( (\Delta\lambda)\rm{grad}_g \log \lambda, X \big).
   \end{split}
  \end{equation}
  Substituting (\ref{5-15-7}) and (\ref{5-15-8}) into
  (\ref{5-15-4}), we obtain
     \begin{equation} \label{5-15-9}
  \begin{split}
 & (4-n)\lambda |\rm{grad}_g \log \lambda|^2\rm{grad}_g \log \lambda
   + (4-n)\lambda {}^M\!\nabla_{\rm{grad}_g \log\lambda} \rm{grad}_g \log \lambda\\
  & \quad  +\lambda \rm{grad}_g(\Delta \log \lambda) + 2 {}^M\!\rm{Ric}(\rm{grad}_g\log \lambda)
   \\ & \quad  +2 {}^M\!\nabla_{\rm{grad}_g \log\lambda} \rm{grad}_g \log \lambda)
     + \rm{grad}_g \big( | \rm{grad}_g \log \lambda |^2\big)
    \\ & \quad   -2 \Delta(\log\lambda)\rm{grad}_g \log \lambda
                +\Big(\lambda | \rm{grad}_g \log \lambda |^2
                  +\lambda ( \Delta \log \lambda)\Big) \rm{grad}_g \log
                  \lambda =0,
  \end{split}
  \end{equation}
  which implies
      \begin{equation} \notag
  \begin{split}
 & \Big((5-n)\lambda |\rm{grad}_g \log \lambda|^2
             +(\lambda-2)(\Delta \log \lambda) \Big) \rm{grad}_g \log \lambda
 \\ &  +\Big( (4-n)\lambda +2)\Big) {}^M\!\nabla_{\rm{grad}_g \log\lambda} \rm{grad}_g \log \lambda
         + \rm{grad}_g \big( | \rm{grad}_g \log \lambda |^2\big)
 \\ & \quad  + \lambda \rm{grad}_g(\Delta \log \lambda)
              + 2 {}^M\!\rm{Ric}(\rm{grad}_g\log \lambda)=0,
    \end{split}
  \end{equation}
as claimed.
\end{proof}

Furthermore, applying the following identity
    \begin{equation}
  {}^M\!\nabla_{\rm{grad}_g \log\lambda} \rm{grad}_g \log \lambda
         =\frac{1}{2}\rm{grad}_g \big( | \rm{grad}_g \log \lambda |^2\big)
  \end{equation}
since
   \begin{equation}
  \begin{split}
  {}^M\!\nabla_{\rm{grad}_g \log\lambda} \rm{grad}_g \log \lambda
     &= \sum\limits_{i,j=1}^{m}g \big(
        {}^M\!\nabla_{e_j(\log \lambda)e_j} \rm{grad}_g \log \lambda, e_i \big) e_i
     \\& =\sum\limits_{i,j=1}^{m}e_j(\log \lambda) \rm{Hess}(\log\lambda) (e_j,e_i) e_i
     \\ & =\sum\limits_{i=1}^{m} \rm{Hess}(\log\lambda)
            \left(e_i, \sum\limits_{i}^{m}e_j(\log \lambda )e_j\right) e_i
     \\ &=\sum\limits_{i=1}^{m} \rm{Hess}(\log\lambda) \big(e_i, \rm{grad}_g \log \lambda \big) e_i
     \\ & =\sum\limits_{i=1}^{m} g( \nabla_{e_i}\rm{grad}_g \log\lambda,  \rm{grad}_g\log \lambda \big) e_i
     \\&=\frac{1}{2}\sum\limits_{i=1}^{m} e_i\Big( g(\rm{grad}_g\log\lambda, \rm{grad}_g \log \lambda)\Big) e_i
     \\ & =\frac{1}{2}\rm{grad}_g \big( | \rm{grad}_g \log \lambda |^2\big),
    \end{split}
  \end{equation}
we have the following consequence.

\begin{corollary}(\cite{OND}) \label{5-16-1}\quad Let $\phi:
(M^n,g) \to (N^n,g)$\quad $(n\geq 3)$ be a conformal map with
dilation $\lambda =f$. Then $\phi$ is a $f$-bi-harmonic map if and
only if
   \begin{equation} \label{5-16-2}
  \begin{split}
 & \Big((5-n)\lambda |\rm{grad}_g \log \lambda|^2
             +(\lambda-2)(\Delta \log \lambda) \Big) \rm{grad}_g \log \lambda
 \\ &  +\Big( \frac{4-n}{2}\lambda +2\Big) \rm{grad}_g \big( | \rm{grad}_g \log \lambda |^2\big)
 \\ & \quad  + \lambda \rm{grad}_g(\Delta \log \lambda)
              + 2 {}^M\!\rm{Ric}(\rm{grad}_g\log \lambda)=0.
    \end{split}
  \end{equation}
  \end{corollary}

\section{The behavior of bi-$f$-harmonic maps from or into singly WPM }

In this section, we will use several special maps from or into WPM
to study bi-$f-$harmonicity like the method in \cite{PK,BMO,Lu1}.

However, if we consider the doubly WPM to discuss the
bi-$f$-harmonic maps by following the method of \cite{Lu1} , then we
will encounter a dreadful trouble \cite{Lu2}. For example, we
consider the inclusion map
  \[i_{y_0} : (M,g) \to M\times_{(\mu,\lambda)} N,\quad i_{y_0}(x)=(x, y_0), \]
since by Equations (7) and (8) in \cite{Lu1}, we know that
   \begin{equation}\label{5.1-1}
     \tau(i_{y_0})=-\frac{m}{2}(0_1, \rm{grad}_h \mu^2)
     \Big|_{i_{y_0}},
     \end{equation}
  \begin{equation}\label{5.1-2}
  \tau_f(i_{y_0})=-\frac{m}{2}f(0_1,\,\mathrm{grad}_h\mu^2)
    +(\mathrm{grad}_{g}f,0_2).
    \end{equation}
Since $\tau_f(i_{y_0})$ contains $f=f(x)$, in order to calculus
$\tau_{f,2}(i_{y_0})$, we must compute the term\\
 ``$\overline\nabla_{(e_j,0)}\overline\nabla_{(e_j,0)}f(0_1,\,\mathrm{grad}_h\mu^2) $"
( see (\ref{3-24-08})). By using (\ref{5-6-1}),  we first obtain
    \[ \overline\nabla_{(e_j,0)}
f(0_1,\,\mathrm{grad}_h\mu^2)=-e_j(f)(0_1,\rm{grad}_h\mu^2)
  +\frac{1}{2\mu^2}|\rm{grad}_h\mu^2|_h^2\ f(e_j,0_2). \]
If again, we will fall into exponential growth terms. A way to avoid
this trouble is that we should restrict the term with $f$ to occur
and let $\mu=1$. This implies that we'd better consider singly WPM
but not doubly WPM.

 Under the assumption that there exist some non-trivial
bi-$f$-harmonic maps, we derive some behavior characteristics on
bi-$f$-harmonic maps.

\subsection{$f$-Bi-harmonicity of the inclusion maps}

We present some non-existence results for bi-$f$-harmonicity of
inclusion maps $i_{y_0}$ of $M$ and $i_{x_0}$ of $N$ under the
singly warped product case. Firstly, we consider the inclusion map
$i_{y_0}: (M,g) \to M\times_\lambda N,\quad i_{y_0}(x)=(x, y_0)$ for
any $y_0 \in N$.

  \begin{theorem} \label{5-18-0}
  The inclusion map
  \[i_{y_0} : (M,g) \to M\times_\lambda N\]
is a non-trivial bi-$f$-harmonic map if and only if $\lambda$ and
$f$ simultaneously satisfy
      \begin{equation} \label{5-18-1}
      \begin{array}{ll}
     & 2f(\rm{Tr}_g \,^M\!\nabla^2 \rm{grad}_{g}f+{}^M\!Ric(\rm{grad}_{g}f)
     +\rm{grad}_{g}(|\rm{grad}_{g}f|^2) =0,
     \end{array}
    \end{equation}
where $f: M \to \mathbb{R}$ is a smooth positive and non-constant
function.
\end{theorem}

  \begin{proof} Let $\{e_j\}_{j=1}^m$ be an orthonormal frame on $M$.
Then from \ref{3-24-08}, bi-$f$-harmonic map of $i_{y_0}$ is
    \begin{equation} \label{4.2.1-3}
      \begin{array}{ll}
   \tau_{f,2}(i_{y_0})&=-f[Tr_{g} (\nabla ^{i_{y_0}})^2 \tau_f(i_{y_0})
   + Tr_{g} \bar R(di_{y_0}, \tau_f(i_{y_0})\,)di_{y_0}] - \nabla^{i_{y_0}}_{grad_{g}f} \tau_f(i_{y_0})\\
   & = -f \sum\limits_{j=1}^m \{ [\nabla^{i_{y_0}}_{e_j}\nabla^{i_{y_0}}_{e_j}
        - \nabla^{ i_{y_0}}_{\,^M\!\nabla_{e_j} e_j } ]\tau_f(i_{y_0})\\
   &   \hskip 1cm + \bar R\big(\tau_f(i_{y_0 }),(e_j,0_2)\big)(e_j,0_2) \}
          -\overline \nabla_{(\rm{grad}_{g} f,0) } \tau_f(i_{y_0}).
   \end{array}
 \end{equation}
Since (\ref{5.1-1}) and (\ref{5.1-2}) with $\mu=1$ give
  \begin{equation} \notag
   \tau(i_{y_0})=0, \quad  \tau_f(i_{y_0})=(\mathrm{grad}_{g} f, 0_2),
 \end{equation}
using (\ref{5-6-1}) and (\ref{5-6-3}), we have
\begin{align} \notag
  & \nabla^{i_{y_0}}_{e_j} \tau_f(i_{y_0})
    =\overline\nabla_{(e_j,0_2)} (\rm{grad}_{g} f ,0_2\,)
    =(\,^M\!\nabla_{e_j}\rm{grad}_{g}f,0_2\,),\\ \notag
& \nabla^{i_{y_0}}_{e_j}\nabla^{i_{y_0}}_{e_j}\tau_f(i_{y_0})
            =\overline\nabla_{(e_j,0)}(\,^M\!\nabla_{e_j}\rm{grad}_{g}f,0_2\,)
             =(\,^M\!\nabla_{e_j}\,^M\!\nabla_{e_j}\rm{grad}_{g}f,0_2\,),\\ \notag
&\overline \nabla_{(\rm{grad}_{g} f,0_2) } \tau_f (i_{y_0})
   =(\,^M\!\nabla_{\rm{grad}_{g}f}\rm{grad}_{g}f,0_2\,)
   =(\frac{1}{2}\rm{grad}_{g}(|\rm{grad}_{g}f|^2),0_2), \\
\notag
 & \nabla^{ i_{y_0}}_{ \,^M\!\nabla_{e_j}e_j } (\tau_f(i_{y_0}))
   =\overline \nabla _{(\,^M\!\nabla_{e_j}e_j ,0_2)}
   (\rm{grad}_{g}f,0_2\,)
  =(\,^M\!\nabla_{\,^M\!\nabla_{e_j}e_j}\rm{grad}_{g}f,0_2\,),
  \\ \notag &  \sum\limits_{j=1}^m\overline R\big( \tau_f(i_{y_0}),(e_j,0_2) \big)(e_j,0_2)
   =\big(\sum\limits_{j=1}^mR^M(\rm{grad}_{g}f,e_j)e_j,0_2\big)=\big({}^M\!Ric(\rm{grad}_{g}f),0_2\big).
   \end{align}
Thus we obtain
 \begin{equation} \label{4.2.1-4}
  \begin{array}{ll}
    \tau_{f,2}(i_{y_0})
=-f\big(\rm{Tr}_g \,^M\!\nabla^2 \rm{grad}_{g}f
+{}^M\!Ric(\rm{grad}_{g}f)
     -\frac{1}{2}\rm{grad}_{g}(|\rm{grad}_{g}f|^2) , 0_2\big),
   \end{array}
    \end{equation}
from which, we conclude that $i_{y_0}$ is a non-trivial
bi-$f$-harmonic map ($\tau_{f,2}(i_{y_0})=0$) if and only if
      \begin{equation} \label{4.2.1-5}
      \begin{array}{ll}
   & 2f(\rm{Tr}_g \,^M\!\nabla^2 \rm{grad}_{g}f
+{}^M\!Ric(\rm{grad}_{g}f)
     +\rm{grad}_{g}(|\rm{grad}_{g}f|^2) =0.
     \end{array}
    \end{equation}
 \end{proof}

\begin{remark}
   If PDE (\ref{5-18-1}) has a solution
   besides $f=const$, then we really find a non-trivial bi-$f$-harmonic map
   which is usual harmonic ($\tau(i_{y_0})=0$) but not $f$-harmonic ($\tau_{f,2}( i_{y_0})\neq0$).
This is a very interesting phenomenon that bi-$f$-harmonic map is
only an extension to $f$-harmonic map but not harmonic map.
\end{remark}

For inclusion map $i_{x_0}: (N,h) \to M\times_\lambda N,
i_{x_0}(y)=(x_0,y)$ for any $x_0 \in M$. Note that $i_{x_0}$ is no
longer harmonic like $i_{y_0}$, we first give the bi-$f$-tension
field of $i_{x_0}$.

 \begin{theorem} \label{5-19-1}
  Let $f:(N,h) \to (0,+\infty)$ be a smooth function. The bi-$f$-tension
field of the inclusion map $i_{x_0}: (N,h) \to M\times_\lambda N$
 is given by
  \begin{equation} \label{5-19-2}
  \begin{array}{ll}
    \tau_{f,2}(i_{y_0})
   & =\Big( (\frac{n+2}{2}f(y)\Delta_N f(y)+\frac{n+1}{2}|\rm{grad}_{h}f(y)|^2 \big)\rm{grad}_g \lambda^2
      -\frac{n^2}{8}f^2(y) \mathrm{grad}_g (|\mathrm{grad}_g \lambda^2|^2,0_2 \Big)
     \\ & \quad   + \big(0_1, \frac{3n+1}{4\lambda^2}f(y)|\rm{grad}_g\lambda^2|^2 \rm{grad}_h f(y)
          -f(y){}^N\! {Ric}(\mathrm{grad}_g f(y)
          \big)\big|_{i_{x_0}}.
   \end{array}
    \end{equation}
\end{theorem}

\begin{proof} Let $\{\bar e_\alpha\}_{\alpha=1}^n$ be an orthonormal frame on $N$.
Then from (\ref{3-24-08}), bi-$f$-harmonic map of $i_{x_0}$ is
    \begin{equation} \label{5-19-3}
      \begin{array}{ll}
   \tau_{f,2}(i_{x_0})& = -f \sum\limits_{\alpha=1}^n \Big( \big(\nabla^{i_{x_0}}_{\bar e_\alpha}
            \nabla^{i_{x_0}}_{\bar e_\alpha}
        - \nabla^{ i_{x_0}}_{\bar \!\nabla_{\bar e_\alpha} \bar e_\alpha } \big)\tau_f(i_{x_0})\\
   &   \hskip 1cm + \bar R\big(\tau_f(i_{x_0 }),(0_1,\bar e_\alpha)\big)(0_1,\bar e_\alpha)\Big)
          -\overline \nabla_{(\rm{grad}_{g} f,0) } \tau_f(i_{x_0}).
   \end{array}
 \end{equation}
Since
  \begin{equation} \label{5-19-4}
  \begin{split}
   \tau(i_{x_0})&=\rm{Tr}_h \nabla dx_{x_0}=\sum\limits_{\alpha=1}^n \{ \big (\bar \nabla_{(0_1,\bar e_\alpha)}
            (0_1,\bar e_\alpha)-(0_1,{}^N\nabla_{\bar e_\alpha} \bar e_\alpha ) \big)
          \\ &  =(-\frac{n}{2} \rm{grad}_{g} \lambda^2,0_2)\circ i_{x_0}, \qquad ( by\ (\ref{5-6-1})\,)
 \end{split}
 \end{equation}
(\ref{1.1}) gives
     \begin{equation} \label{5-19-5}
  \begin{split}
   \tau_f(i_{x_0})&=-\frac{n}{2}f(y)(\rm{grad}_{g} \lambda^2,0_2)\circ i_{x_0}+
                      (0_1,\rm{grad}_{h} f(y))\circ i_{x_0}.
 \end{split}
 \end{equation}
Thus by (\ref{5-6-1}) we have
\begin{equation} \notag
\begin{split}
   \nabla^{i_{x_0}}_{\bar e_\alpha} \tau_f(i_{x_0})
   & =-\frac{n}{2}\overline\nabla_{(\bar e_\alpha,0_2)} f(y)(\rm{grad}_{g} \lambda^2,0_2)\circ i_{x_0}
       +\overline\nabla_{(\bar e_\alpha,0_2)} (0_1,\rm{grad}_{h}f(y)) \circ i_{x_0}
 \\ &  =-\frac{n+1}{2} \bar e_\alpha(f(y)) (\rm{grad}_g \lambda^2,0_2)
      -\frac{n}{4\lambda^2} f(y)|\rm{grad}_g\lambda^2|^2 (0_1,\bar e_\alpha)
      \\ & \quad +(0_1,  \,^N\!\nabla_{\bar e_\alpha}{grad}_{h}f)\big|_{i_{x_0}},
 \end{split}
 \end{equation}
 \vskip -5mm
 \begin{equation} \notag
\begin{split}
   \nabla^{i_{x_0}}_{{}^N\nabla_{\bar e_\alpha} \bar e_\alpha} \tau_f(i_{x_0})
   &   =-\frac{n+1}{2} {}^N\nabla_{\bar e_\alpha} \bar e_\alpha(f(y))
        (\rm{grad}_g \lambda^2,0_2)\big|_{i_{x_0}}
   \\ & \quad    -\frac{n}{4\lambda^2} f(y)|\rm{grad}_g\lambda^2|^2
              (0_1,{}^N\nabla_{\bar e_\alpha} \bar e_\alpha)
      +(0_1,  \,^N\!\nabla_{{}^N\nabla_{\bar e_\alpha} \bar e_\alpha}{grad}_{h}f)\big|_{i_{x_0}},
 \end{split}
 \end{equation}
 \vskip -5mm
 \begin{equation} \label{5-19-5a}
\begin{split}
   \nabla^{i_{x_0}}_{ \rm{grad}_h f(y)} \tau_f(i_{x_0})
   &  =-\frac{n+1}{2} \rm{grad}_h f(y)(f(y)) (\rm{grad}_g \lambda^2,0_2)
    \\ &   -\frac{n}{4\lambda^2} f(y)|\rm{grad}_g\lambda^2|^2 (0_1,{grad}_{h}f(y))
      +(0_1,  \,^N\!\nabla_{\rm{grad}_h f(y)} \rm{grad}_h f(y))\big|_{ i_{x_0}}
    \\  &  =-\frac{n+1}{2} |{grad}_{h}f(y)|^2 (\rm{grad}_g \lambda^2,0_2)
         +(0_1,  \frac{1}{2} \rm{grad}_h (|\rm{grad}_h f(y)|^2)
    \\ &\quad   -\frac{n}{4\lambda^2} f(y)|\rm{grad}_g\lambda^2|^2
                (0_1,\bar e_\alpha(\rm{grad}_h f(y)\bar e_\alpha)\big|_{ i_{x_0}},
 \end{split}
 \end{equation}
\vskip -3mm
\begin{equation} \notag
\begin{split}
   \nabla^{i_{x_0}}_{\bar e_\alpha}\nabla^{i_{x_0}}_{\bar e_\alpha} \tau_f(i_{x_0})
   & =\Big( -\frac{n+1}{2} \bar e_\alpha\big(\bar e_\alpha ( f(y))\big)
         -\frac{1}{2} {}^N\!\rm{Hess}(f(y))(\bar e_\alpha,\bar
         e_\alpha)\Big) (\rm{grad}_{g} \lambda^2,0_2)
   \\ & \quad +\frac{n}{8\lambda^2}f(y)|\rm{grad}_g\lambda^2|^2(\rm{grad}_{g} \lambda^2,0_2)
              -\frac{2n+1}{4\lambda^2}|\rm{grad}_g\lambda^2|^2 (0_1,\bar e_\alpha(f(y))\bar e_\alpha)
     \\ &\quad -\frac{n}{4\lambda^2}f(y)|\rm{grad}_g\lambda^2|^2
        (0_1, {}^N\!\nabla_{\bar e_\alpha}\bar e_\alpha)
      +(0_1,  \,^N\!\nabla_{\bar e_\alpha}{grad}_{h}f)\big|_{i_{x_0}},
 \end{split}
 \end{equation}
      \begin{equation} \label{5-19-5b}
\begin{split}
   \rm{Tr}_h (\nabla^{i_{x_0}})^2 \tau_f(i_{x_0})
   & =-\frac{n+2}{2} (\Delta_N f(y)(\rm{grad}_{g} \lambda^2,0_2)
    +\frac{n^2}{8\lambda^2}f(y)|\rm{grad}_g\lambda^2|^2(\rm{grad}_{g} \lambda^2,0_2)
   \\ & \quad           -frac{2n+1}{4\lambda^2}|\rm{grad}_g\lambda^2|^2
   \big(0_1,\rm{grad}_h f(y)\big)\big|_{i_{x_0}}.
 \end{split}
 \end{equation}
On the other hand, by (\ref{5-6-3}) we have
     \begin{equation} \label{5-19-5c}
      \begin{array}{ll}
   & \sum\limits_{\alpha=1}^n\bar R\big(\tau_f(i_{x_0 }),(0_1,\bar e_\alpha)\big)(0_1,\bar e_\alpha)
  \\ & =\frac{n^2}{4} f(y)\left(\,^{M}\!\nabla_{\mathrm{grad}_g \lambda^2} \mathrm{grad}_g \lambda^2
             -\frac{1}{2\lambda^2} \mathrm{grad}_g \lambda^2(\lambda^2)\mathrm{grad}\ \lambda^2,  0_2\right)
   \\ & \quad +\left(0_1, \sum\limits_{\alpha=1}^n R^N(\rm{grad}_g f(y), \bar e_\alpha)\bar e_\alpha\right)
        \big|_{i_{x_0}}
   \\ & = \frac{n^2}{8}f(y)(\mathrm{grad}_g (|\mathrm{grad}_g \lambda^2|^2),0_2)
        -\frac{n^2}{8\lambda^2}f(y)|\mathrm{grad}_g \lambda^2|^2 (\mathrm{grad}_g \lambda^2,0_2)
    \\ & \qquad +\big(0_1,{}^N\! {Ric}(\mathrm{grad}_g f(y))\big)\big|_{i_{x_0}}.
   \end{array}
 \end{equation}
Substituting (\ref{5-19-5a}), (\ref{5-19-5b}) and (\ref{5-19-5c})
into (\ref{5-19-3}), we obtain
 \begin{equation} \notag
  \begin{array}{ll}
    \tau_{f,2}(i_{y_0}) & =\Big( (\frac{n+2}{2}f(y)\Delta_N f(y)+\frac{n+1}{2}|\rm{grad}_{h}f(y)|^2 \big)\rm{grad}_g \lambda^2
       -\frac{n^2}{8}f^2(y)\mathrm{grad}_g (|\mathrm{grad}_g \lambda^2|^2,0_2 \Big)
    \\ & \quad + \big(0_1, \frac{3n+1}{4\lambda^2}f(y)|\rm{grad}_g\lambda^2|^2 \rm{grad}_h f(y)
          -f(y){}^N\! {Ric}(\mathrm{grad}_g f(y) \big)\big|_{i_{x_0}},
   \end{array}
    \end{equation}
as claimed.
 \end{proof}

\begin{corollary} \label{5-19-6}
  Let $f:(N,h) \to (0,+\infty)$ be a smooth function. The inclusion map $i_{x_0}: (N,h) \to M\times_\lambda N$
 is bi-$f$-harmonic map if and only if $\lambda$ and $f$ satisfy
  \begin{equation} \label{5-19-7}
   \left\{ \begin{array}{ll}
      & \Big( 4(n+2)f(y)\Delta_N f(y)+4(n+1)|\rm{grad}_hf(y)|^2 \big)\rm{grad}_g \lambda^2
      -n^2 f^2(y) \mathrm{grad}_g (|\mathrm{grad}_g \lambda^2|^2\big|_{i_{x_0}}=0,
      \\ &  (3n+1)f(y)|\rm{grad}_g\lambda^2|^2 \rm{grad}_h f(y)
          -4f(y)\lambda^2 {}^N\! {Ric}(\mathrm{grad}_g f(y)) \big|_{i_{x_0}}=0.
   \end{array}
   \right.
    \end{equation}
\end{corollary}

\begin{corollary} \label{5-19-8} If $x_0$ is a critical point of $\rm{grad}_g \lambda^2$
 but not a critical point of $\lambda^2$ and, ${}^N\! {Ric}(\mathrm{grad}_g f(y))=0$ but
 $\mathrm{grad}_g f(y)\neq 0$, then the inclusion map $i_{x_0}: (N,h) \to M\times_\lambda N$
 is non-trivial bi-$f$-harmonic map.
\end{corollary}

\subsection{Bi-$f$-harmonicity of the projection maps} \label{6.2bi-f}

 In this subsection, we attempt two methods so-called projection maps
related to singly WPM to discuss bi-$f$-harmonic maps. We first give
two lemmas.

 \begin{lemma} \label{4.2.3-01}
For given projection map
 \[  \bar\pi_1: M \times_\lambda N \to M, \quad \bar\pi_1(x, y) = x, \]
let $f: M\times_\lambda N \to (0,+\infty)$ be smooth function. Then
bi-$f$-bitension field of $\bar\pi_1$ is
     \begin{equation} \label{4.2.3-2}
      \begin{array}{ll}
   \tau_{f,2}(\overline\pi_1)& =-f \rm{Tr}_g (\,^M\!\nabla^2)f\cdot \rm{grad}_{ g} \log(f\lambda^n)
             -f {}^M\!\rm{Ric}\big(\rm{grad}_{ g} f\ \log(f\lambda^n) \big)\\
        & \hskip 3mm -fn \,^M\!\nabla_{\rm{grad}_{  g}\log \lambda} f\cdot \rm{grad}_{g} \log(f\lambda^n)
       -\,^M\!\nabla_{\rm{grad}_{g} f} f\cdot \rm{grad}_{ g} \log(f\lambda^n).
     \end{array}
     \end{equation}
\end{lemma}

\begin{proof} Let $\{e_j\}_{j=1}^m $ and $\{\bar e_\alpha\}_{\alpha=1}^n$
be local orthonormal frame fields on $(M,g)$ and $(N,h)$,
respectively. Then $\{(e_j,0_2),(0_1,\frac{1}{\lambda} \bar
e_\alpha)\}_{j=1,\ldots,m,\alpha=1,\ldots,n}$ is a local orthonormal
frame on $M^m \times_{\lambda} N^n$. By a similar calculation as
(\ref{5-19-4}) and (\ref{5-19-5}), we have
  \begin{equation} \label{4.2.3-3}
  \begin{array}{ll}
   \tau(\bar\pi_1)&= Tr_{\bar g} \nabla d\bar\pi_1
          =\sum\limits_{j=1}^m \big( \,^M\!\nabla_{d\bar\pi_1(e_j,0_2)}d\bar\pi_1(e_j,0_2)
                 -d\bar\pi_1(\,\bar\nabla_{(e_j,0_2)}(e_j,0_2)\,) \big)\\
      & \quad + \frac{1}{\lambda^2}\sum\limits_{\alpha=1}^n \big(
     \,^M\!\nabla_{d\bar\pi_1(0_1,\bar e_\alpha)}d\bar\pi_1(0_1,\bar e_\alpha)
                    -d\bar\pi_1(\,\bar\nabla_{(0_1,\bar e_\alpha)}(0_1,\bar e_\alpha)\,)\,\big)\\
      &=\frac{n}{2\lambda^2}\rm{grad}_{g} \lambda^2 \mid \bar\pi_1\\
      &= n \rm{grad}_{g} \log\lambda \mid \bar\pi_1,
  \end{array}
 \end{equation}

 \begin{equation} \label{4.2.3-4}
  \begin{array}{ll}
   \tau_f(\bar\pi_1)&= fn\rm{grad}_{g}\log \lambda\circ \bar\pi_1
                +d\bar\pi_1(\rm{grad}_{g}f,\frac{1}{\lambda^2}\rm{grad}_{h}f)\\
                     &=f \rm{grad}_g\log(\lambda^nf)\mid \bar\pi_1.
   \end{array}
 \end{equation}
Thus we have
   \begin{equation} \notag
  \begin{array}{ll}
    &\nabla^{\bar\pi_1}_{(e_j,0_2)}\tau_f(\bar\pi_1)
       =\,^M\!\nabla_{d\bar\pi_1(e_j,0)} \tau_f(\bar\pi_1)
      = \,^M\!\nabla_{e_j} f\cdot \rm{grad}_{g} \log(f\lambda^n),
   \end{array}
 \end{equation}
\vskip -6mm
  \begin{equation} \notag
  \begin{array}{ll}
   & \nabla^{ \bar\pi_1}_{ \overline{\nabla}_{e_j,0_2)}(e_j,0_2) } \tau_f(\bar\pi_1)
    =\nabla^{ \bar\pi_1}_{(\,^M\!\nabla_{e_j},0_2)} \tau_f(\bar\pi_1)
    =\,^M\!\nabla_{\,^M\!\nabla_{e_j}e_j } f\cdot \rm{grad}_{g} \log(f\lambda^n),
     \end{array}
 \end{equation}
\vskip -5mm
  \begin{equation} \notag
  \begin{array}{ll}
   & \nabla^{\bar\pi_1}_{(e_j,0_2)}\nabla^{\bar\pi_1}_{e_j,0_2)}\tau_f(\bar\pi_1)
     =\nabla^{\bar\pi_1}_{(e_j,0_2)} \,^M\!\nabla_{e_j} f\cdot \rm{grad}_{\mu^2g} \log(f\lambda^n)
      \\      &  \hskip 3cm  = \,^M\!\nabla_{e_j}\,^M\!\nabla_{e_j} f\cdot \rm{grad}_{g} \log(f\lambda^n),
   \end{array}
 \end{equation}

    \begin{equation} \notag
  \begin{array}{ll}
   & \nabla^{\bar\pi_1}_{\rm{grad}_{\bar g} f } \tau_f (\bar\pi_1)
       =\nabla^{\bar\pi_1}_{(\rm{grad}_g f,0_2)+(0_1,\frac{1}{\lambda^2} \rm{grad}_h f) } \tau_f (\bar\pi_1)
   \\ & \hskip 3cm  =\,^M\!\nabla_{\rm{grad}_{g} f } f\cdot \rm{grad}_{g} \log(f\lambda^n),
     \end{array}
 \end{equation}

 \begin{equation} \notag
  \begin{array}{ll}
    \nabla^{\bar\pi_1}_{(0_1,\frac{1}{\lambda} \bar{e}_\alpha) }\tau_f(\bar\pi_1)
         =\,^M\!\nabla_{0_1 }\tau_f(\bar\pi_1)
         =0,
   \end{array}
 \end{equation}

 \begin{equation} \notag
  \begin{array}{ll}
    \nabla^{\bar\pi_1}_{\overline{\nabla}_{(0_1,\frac{1}{\lambda} \bar{e}_\alpha)}(0_1,\frac{1}{\lambda} \bar{e}_\alpha) }\tau_f(\bar\pi_1)
        & =\frac{1}{\lambda^2} \nabla^{ \bar\pi_1}_{(0_1,\,^N\!\nabla_{\bar{e}_\alpha} \bar{e}_\alpha )
          -\frac{1}{2}(\rm{grad}_{ g}\lambda^2,0_2 )} \tau_f(\bar\pi_1)\\
    & =-\,^M\!\nabla_{\rm{grad}_{g}\log\lambda} f\cdot \rm{grad}_{g} \log(f\lambda^n),
   \end{array}
 \end{equation}

 \begin{equation} \notag
  \begin{array}{ll}
    \nabla^{\bar\pi_1}_{(0_1,\frac{1}{\lambda} \bar{e}_\alpha)}
        \nabla^{\bar\pi_1}_{(0_1,\frac{1}{\lambda} \bar{e}_\alpha)}\tau_f(\bar\pi_1)=0,
  \end{array}
 \end{equation}
which imply that
   \begin{equation} \label{4.2.3-5}
    \begin{array}{ll}
    Tr_{\bar g}(\nabla^{\bar\pi_1})^2 \tau_f(\bar\pi_1)
    & = Tr_g (\,^M\!\nabla^2)f\cdot \rm{grad}_{g} \log(f\lambda^n)\\
        & \hskip 2cm +n \,^M\!\nabla_{\rm{grad}_{  g}\log \lambda} f\cdot \rm{grad}_{ g} \log(f\lambda^n).
     \end{array}
 \end{equation}
On the other hand, since
 \begin{equation} \label{4.2.3-6}
  \begin{array}{ll}
    Tr_{\bar g} R^M(d\bar \pi_1, \tau_f(\bar\pi_1))d\bar\pi_1
    &= \sum\limits_{j=1}^m  R^M\big(d\pi_1(e_j,0_2), \tau_f(\bar\pi_1) \big) d\pi_1(e_j,0_2)\\
    & \hskip 1cm +\sum\limits_{\alpha=1}^n R^M\big(
    d\pi_1(0_1,\frac{1}{\lambda}\bar {e}_\alpha), \tau_f(\bar\pi_1)\big)
            d\pi_1(0_1,\frac{1}{\lambda}\bar {e}_\alpha\, )\\
        & =\sum\limits_{j=1}^m R^M\big( e_j, \rm{grad}_{ g} f\ \log(f\lambda^n)\big) e_j\\
        &   = -{}^M\!\rm{Ric}\big(\rm{grad}_{ g} f\ \log(f\lambda^n) \big),
     \end{array}
    \end{equation}
we get
     \begin{equation} \notag
      \begin{array}{ll}
   \tau_{f,2}(\overline\pi_1)&=-f\big( \rm{Tr}_{\bar g}(\nabla^{\bar\pi_1})^2 \tau_f(\bar\pi_1)
         +f \rm{Tr}_{\bar g} R^M(d\bar \pi_1, \tau_f(\bar\pi_1))d\bar\pi_1 \big)
           -\nabla^{\bar\pi_1}_{\rm{grad}_{\bar g} f } \tau_f(\bar\pi_1)\\
      & =-f \rm{Tr}_g (\,^M\!\nabla^2)f\cdot \rm{grad}_{ g} \log(f\lambda^n)
             -f {}^M\!\rm{Ric}\big(\rm{grad}_{ g} f\ \log(f\lambda^n) \big)\\
        & \hskip 3mm -fn \,^M\!\nabla_{\rm{grad}_{  g}\log \lambda} f\cdot \rm{grad}_{g} \log(f\lambda^n)
       -\,^M\!\nabla_{\rm{grad}_{g} f} f\cdot \rm{grad}_{ g} \log(f\lambda^n).
     \end{array}
     \end{equation}
Thus we complete the proof.
 \end{proof}

For another projection map $\bar\pi_2 :M\times_\lambda N \to N$,
note that at this time there are some differences between
$\bar\pi_1$ and $\bar\pi_2$, that is $\tau(\bar\pi_2)$ and
$\tau_f(\bar\pi_2)$ respectively satisfy
     \begin{align}
    \notag &\tau(\bar\pi_2)=0, \\ \notag
    &  \tau_f(\bar\pi_2)=\frac{1}{\lambda^2} \rm{grad}_h  f\mid
    \bar\pi_2,
\end{align}
 we have
 \begin{lemma} \label{4.2.3-01b}
Given a projection map $\bar\pi_2: (M \times_\lambda N \to (N, h),\
\bar\pi_2(x, y) = y$, let $f: M\times_\lambda N \to (0,+\infty)$ be
smooth function. Then bi-$f$-tension field of $\bar\pi_2$ is
     \begin{equation} \label{6.16-54.2.3-2b}
   \begin{array}{ll}
    \tau_{f,2}(\overline\pi_2)& =-\frac{f}{\lambda^2} Tr_h\,^N\!\nabla^2 \rm{grad}_h f
      -\frac{f}{\lambda^2} {}^N\!\rm{Ric}\big(\rm{grad}_h  f )
      \\ & \hskip 2cm          -\frac{1}{2\lambda^2}\rm{grad}_h f (|\rm{grad}_h f|^2).
  \end{array}
   \end{equation}
  \end{lemma}

From Lemmas \ref{4.2.3-01} and \ref{4.2.3-01b}, we easily conclude
that

 \begin{corollary} \label{4.2.3-01ab}
(i) If $\lambda$ and $f$ are non-constant function and $\rm{grad}_{
g} \log(f\lambda^n)\circ \bar\pi_1=0$, then $\bar \pi_1$ is a
non-trivial bi-$f$-harmonic map.

\noindent (ii) Suppose $\lambda$ and $f$ are non-constant function.
If $\rm{grad}_h$ is non-zero constant and
${}^N\!\rm{Ric}\big(\rm{grad}_h  f )=0$, then $\bar\pi_2$ is a
non-trivial bi-$f$-harmonic map.
  \end{corollary}

\subsection{Bi-$f$-harmonicity of the product maps with harmonic
factor} \label{6.2bi-pm}

Now, we turn to consider a type of product map such as
\begin{equation} \notag
\overline{\Phi}=\varphi_M \times \varphi_N:  M\times_\lambda N
 \to (M\times N, g\oplus h)
 \end{equation}
 defined by
\[ \varphi_M \times \varphi_N(x, y) = (\varphi_M(x),\varphi_N(y)),\]
 where $\varphi_M: M \to M $ and $\varphi_N: N \to N$ are smooth maps.
 In order to get some interesting results, we usually make some
restrictions for $\varphi_M$ and $\varphi_N$. Since in advance we
observe that $\tau(\overline{\Phi})$ contains $\tau(\varphi_M)$ and
$\tau(\varphi_N)$ (see (\ref{4.2.4-2}),  typically, $\varphi_M$ and
$\varphi_N$ should be chosen as harmonic maps so that
$\tau(\overline\Phi_M)$ has a simpler form. Thus we have

 \begin{proposition} \label{4.2.4-1}
 Suppose that
 $\varphi_M: (M,g)\to M,\, \varphi_N: N \to N$ are two harmonic maps. Let the product map
 $\overline\Phi=\varphi_M \times \varphi_N:  M \times_\lambda N \to (M\times N, g\oplus h)$
 be defined by $\overline\Phi(x,y)=(\varphi_M (x),\varphi_N(y))$ and
 $f: M \times_\lambda N \to \mathbb{R}$ smooth positive function . Then the bi-$f$-tension field of
 $\overline\Phi$ is
     \begin{equation} \label{4.2.4-1a}
    \begin{array}{ll}
&   \tau_f(\overline\Phi)
     =\big( d\varphi_M(\tau_{f,2}(\overline\pi_1)),d\varphi_N(\tau_{f,2}(\overline \pi_2)) \big)
\end{array}
 \end{equation}
under some conventions below:
   \begin{equation} \label{4.2.4-1b}
   \begin{split}
   & d\varphi_L( {}^L\nabla_{\cdot}):={}^L\nabla_{d\varphi_L(\cdot)}, \quad L=M, N,\\
      &   d\varphi_L( {}^L\nabla_{\cdot}^2):={}^L\nabla_{d\varphi_L(\cdot)}{}^L\nabla_{d\varphi_L(\cdot)},\\
   &  d\varphi_L({}^L\!R(\tau_f(\overline\pi_i),\cdot )\cdot
          ={}^L\!R(\tau_f(\overline\pi_i),d\varphi_L(\cdot) )
          d\varphi_L(\cdot),\quad i=1,2.
 \end{split}
 \end{equation}
\end{proposition}

\begin{proof} Since $\varphi_M$ and $\varphi_N$ are harmonic, we have $\tau(\varphi_M)=\tau(\varphi_N)=0$.
As the trick of \cite{Lu1}, we have
  \begin{equation} \label{4.2.4-2}
  \begin{array}{ll}
  \tau(\overline\Phi)&= \sum\limits_{j=1}^m
          \big( \nabla_{d\overline\Phi(e_j,0_2)}d\overline\Phi(e_j,0_2)
                 -d\overline\Phi(\,\bar\nabla_{(e_j,0_2)}(e_j,0_2)\,) \big)\\
   & +\sum\limits_{\alpha=1}^n \frac{1}{\lambda^2} \Big(
   \nabla_{d\varphi_{M}\times d\varphi_N(0_1,\bar e_\alpha)}d\varphi_{M}\times d\varphi_N(0_1,\bar e_\alpha)
   \\& \quad - d\varphi_{M}\times d\varphi_N(\,\bar\nabla_{(0_1,\bar e_\alpha)}(0_1,\bar e_\alpha)\,)\Big)\\
   &=(\tau(\varphi_M),0_2)+ n(d\varphi_M(\rm{grad}_{\mu^2g} \log\lambda),0_2\,)\\
 & \qquad +\frac{1}{\lambda^2}(0_1,\tau(\varphi_N)))\\
 &=(n\ d\varphi_M(\rm{grad}_{g} \log\lambda) ,0_2\,).
  \end{array}
 \end{equation}
So
 \begin{equation} \label{4.2.4-3}
  \begin{array}{ll}
   \tau_f(\overline\Phi)&= fn(d\varphi_M(\rm{grad}_{g} \log\lambda) ,0_2\,)
                       +d\varphi_{M}\times d\varphi_N(\rm{grad}_{g}f,\frac{1}{\lambda^2}\rm{grad}_{h}f)\\
                     &=(\,fn d\varphi_M(\rm{grad}_{g}\log{\lambda})+d\varphi_M(\rm{grad}_{g}f),
                                      \frac{1}{\lambda^2}d\varphi_N(\rm{grad}_h
                                      f)\,)\\
                     &=\big(f\ d\varphi_M(\rm{grad}_{g}\log(\lambda^nf)),
                                      \frac{1}{\lambda^2}d\varphi_N(\rm{grad}_h f) \big)
  \end{array}
 \end{equation}
 Next we process $\tau_{f,2}(\overline\Phi)$. To this end, we need to tackle
 two intricate terms by two steps:
 \textbf{Step 1} Consider $Tr_{\bar g}(\nabla^{\bar\Phi})^2
 \tau_f(\bar\Phi)$.
Since
 \begin{equation} \notag
  \begin{array}{ll}
    \nabla^{\bar\Phi}_{(e_j,0_2)}\tau_f(\bar\Phi)
    = (\,^M\!\nabla_{d\varphi_M(e_j)} f\ d\varphi_M(\rm{grad}_{g}\log(\lambda^nf)),\,0_2),
   \end{array}
 \end{equation}

\begin{equation} \notag
  \begin{array}{ll}
    \nabla^{\bar\Phi}_{(e_j,0_2)}\nabla^{\bar\Phi}_{(e_j,0_2)}\tau_f(\bar\Phi)
           =(\,^M\!\nabla_{d\varphi_M(e_j)}\,^M\!\nabla_{d\varphi_M(e_j)}
             f\ d\varphi_M(\rm{grad}_{g}\log(\lambda^nf)),\,0_2),
   \end{array}
 \end{equation}

\begin{equation} \notag
  \begin{array}{ll}
   \nabla^{ \bar\Phi}_{ \overline{\nabla}_{(e_j,0_2)}(e_j,0_2) } \tau_f(\bar\Phi)
    & =\nabla^{ \bar\Phi}_{(\,^M\!\nabla_{e_j},0_2)} \tau_f(\bar\Phi)\\
    & =(\,^M\!\nabla_{d\varphi_M(\,^M\!\nabla_{e_j}e_j) } f  d\varphi_M(\rm{grad}_{g}\log(\lambda^nf))
           ,0_2 )
     \end{array}
 \end{equation}

 \begin{equation} \notag
  \begin{array}{ll}
    \nabla^{\bar\Phi}_{ (0_1,\bar{e}_\alpha) }\tau_f(\bar\Phi)
         = \nabla_{(0_1,d\varphi_N (\bar{e}_\alpha ) }\tau_f(\bar\Phi)
         =(\,0_1,\,^N\!\nabla_{ d\varphi_N(\bar{e}_\alpha) } \frac{1}{\lambda^2}d\varphi_N(\rm{grad}_h f \,),
   \end{array}
 \end{equation}

\begin{equation} \notag
  \begin{array}{ll}
    \nabla^{\bar\Phi}_{(0_1, \bar{e}_\alpha)} \nabla^{\bar\Phi}_{(0_1, \bar{e}_\alpha)}\tau_f(\bar\Phi)
        &=(\,0_1,\,^N\!\nabla_{ d\varphi_N(\bar{e}_\alpha) } \,^N\!\nabla_{ d\varphi_N(\bar{e}_\alpha) }
                   \frac{1}{\lambda^2} d\varphi_N (\,\rm{grad}_h f)\,)
  \end{array}
 \end{equation}

 \begin{equation} \notag
  \begin{array}{ll}
    \nabla^{\bar\Phi}_{\overline{\nabla}_{(0_1,\bar{e}_\alpha)}(0_1, \bar{e}_\alpha) }\tau_f(\bar\Phi)
        & = \nabla^{ \bar\Phi}_{(0_1,\,^N\!\nabla_{\bar{e}_\alpha} \bar{e}_\alpha )
          -\frac{1}{2}(\rm{grad}_{g}\lambda^2,0_2 )} \tau_f(\bar\Phi)\\
    & =(0_1,\,^N\!\nabla_{d\varphi_N(\,^N\!\nabla_{\bar{e}_\alpha} \bar{e}_\alpha) }
        \frac{1}{\lambda^2} d\varphi_N(\,\rm{grad}_{h} f \,) \,)\\
    &    - (\lambda \,^M\!\nabla_{d\varphi_M(\rm{grad}_g \lambda)}
    f\ d\varphi_M(\rm{grad}_{g}\log(\lambda^nf)),\, 0_2\,),
   \end{array}
 \end{equation}
we have
\begin{equation} \label{4.2.4-4}
  \begin{array}{ll}
    \rm{Tr}_{\bar g}(\nabla^{\bar\Phi})^2 \tau_f(\bar\Phi)&
    =\big( \rm{Tr}_g (\,^M\!\nabla_{d\varphi_M})^2\ f  d\varphi_M(\rm{grad}_{g}\log(\lambda^nf))\\
    & \hskip 1cm +n \,^M\!\nabla_{d\varphi_M(\rm{grad}_g \log \lambda)}
          f  d\varphi_M(\rm{grad}_{g}\log(\lambda^nf)),\, 0_2\big)\\
     & +\big( 0_1, \frac{1}{\lambda^4}\rm{Tr}_h (\,^N\!\nabla_{d\varphi_N)})^2
     d\varphi_N (\,\rm{grad}_h f )\big).
  \end{array}
 \end{equation}
\textbf{Step 2} Consider  $\rm{Tr}_{\bar g} R\big(
 d\overline\Phi,\tau_f(\bar\Phi)d\overline\Phi$.
Since
 \begin{equation} \notag
  \begin{array}{ll}
    &  \sum\limits_{j=1}^m R\big(d\varphi_M(e_j),0_2),
           \tau_f(\bar\Phi) \big)(d\varphi_M(e_j),0_2)\\
    &=\big( \sum\limits_{j=1}^m {}^M\!R(d\varphi_M(e_j),
                 d\varphi_M(f\ \rm{grad}_ g \log(f\lambda^n)) )d\varphi_M(e_j), 0_2\big)\\
    & = \big( \rm{Tr}_g {}^M\!R(d\varphi_M, d\varphi_M(f\ \rm{grad}_{g} \log(f\lambda^n) )\,)
         d\varphi_M,\, 0_2\big),
     \end{array}
    \end{equation}

    \begin{equation} \notag
  \begin{array}{ll}
     &\sum\limits_{\alpha=1}^n \frac{1}{\lambda^2}R\big(
     (0_1,d\phi_N(\bar {e}_\alpha)\,), \tau_f(\bar\Phi) \big) (0_1,d\phi_N(\bar {e}_\alpha)\, )\\
      &   =\big(\,0_1,\frac{1}{\lambda^4}\rm{Tr}_h R^N(d\phi_N, \rm{grad}_h f )d\phi_N \,\big),
     \end{array}
    \end{equation}
we obtain
         \begin{equation} \label{4.2.4-5}
  \begin{array}{ll}
     \rm{Tr}_{\bar g} R\big(d\overline\Phi,\tau_f(\bar\Phi)d\overline\Phi
     &= -\big( \rm{Tr}_g {}^M\!R(d\varphi_M(f \rm{grad}_{g} \log(f\lambda^n),d\varphi_M  )\,)d\varphi_M,\, 0_2\big)\\
     &\quad -\big(\,0_1,\frac{1}{\lambda^2} \rm{Tr}_h R^N(d\phi_N(\rm{grad}_h f),d\phi_N )d\phi_N \,\big),
     \end{array}
    \end{equation}
Finally, note that
\begin{equation} \label{4.2.4-6}
  \begin{array}{ll}
    \nabla^{\bar\Phi}_{\rm{grad}_{\bar g} f} \tau_f (\bar\Phi)
    & =(\,^M\!\nabla_{d\varphi_M(\rm{grad}_{g} f )} f\ d\varphi_M(\rm{grad}_{g}
    \log(f\lambda^n)), 0_2)\\
    &  \quad + (0_1,   \frac{1}{\lambda^4}\,^M\!\nabla_{d\varphi_M(\rm{grad}_h f )}
   d\varphi_N (\,\rm{grad}_h f)).
   \end{array}
 \end{equation}
Putting (\ref{4.2.4-4}), (\ref{4.2.4-4}) and (\ref{4.2.4-4})
together, we have
      \begin{equation} \label{4.2.4-7}
      \tau_{f,2}(\overline\Phi)=(A,B),
      \end{equation}
where $A$ and $B$ denote by
      \begin{equation} \label{4.2.4-8a}
 \begin{array}{ll}
     & A= -f\ \rm{Tr}_g (\,^M\!\nabla_{d\varphi_M})^2\ f\ d\varphi_M(\rm{grad}_{g}\log(\lambda^nf))
    \\ & \quad  - f\ \rm{Tr}_g {}^M\!R(d\varphi_M(f\ \rm{grad}_{g} \log(f\lambda^n)),d\varphi_M \,)d\varphi_M\\
   &\quad  - \ \,^M\!\nabla_{d\varphi_M(\rm{grad}_{g} f )} f\ d\varphi_M(\rm{grad}_{g}\log(f\lambda^n)
    \\ & \quad  - nf \,^M\!\nabla_{d\varphi_M(\rm{grad}_g  \log\lambda)}f\ d\varphi_M(\rm{grad}_{g} \log(f\lambda^n)
  \end{array}
 \end{equation}
 and
      \begin{equation} \label{4.2.4-8b}
 \begin{array}{ll}
     & B=-\frac{f}{\lambda^2}\rm{Tr}_h (\,^N\!\nabla_{d\varphi_N)})^2
     d\varphi_N (\,\rm{grad}_h f )
     -\frac{f}{\lambda^2} \rm{Tr}_h R^N(d\phi_N(\rm{grad}_h f),d\phi_N,  )d\phi_N\\
     &\quad
      -\frac{1}{\lambda^2}\,^M\!\nabla_{d\varphi_M(\rm{grad}_h f )} d\varphi_N (\,\rm{grad}_h f).
   \end{array}
 \end{equation}
 Connecting Lemmas \ref{4.2.3-01} and \ref{4.2.3-01b} above,
 under the notation conventions \ref{4.2.4-1b},  (\ref{4.2.4-7}) can
 be singly written as
      \[  \tau_{f,2}(\overline\Phi)
     =\big( d\varphi_M(\tau_{f,2}(\overline\pi_1)),d\varphi_N(\tau_{f,2}(\overline \pi_2)) \big), \]
as claimed.
 \end{proof}

 When $\varphi_M=Id_M$ or $\varphi_N=Id_N$, we easily obtain the following
 propositions.

\begin{proposition} \label{4.2.4-9}
(i)~ $\overline\Phi$ with $\varphi_M=Id_M$ is a bi-$f$-harmonic map
if and only if the projection map $\overline\pi_1$ is
bi-$f$-harmonic and $d\varphi_N(\tau_{f,2}(\overline \pi_2))=0$.

\noindent(ii)~$\overline\Phi$ with $\varphi_N=Id_N$ is a
bi-$f$-harmonic map if and only if the projection map
$\overline\pi_2$ is also and $d\varphi_N(\tau_{f,2}(\overline
\pi_1))=0$.

\noindent(iii)~$\overline\Phi$ with $\varphi_M=Id_M$ and
$\varphi_N=Id_N$ is a bi-$f$-harmonic map if and only if both
$\overline\pi_1$ and $\overline \pi_2$ are also.
\end{proposition}

 \begin{remark}In above propositions, neither $d\varphi_N(\tau_{f,2}(\overline \pi_2))=0$
nor $d\varphi_M(\tau_{f,2}(\overline \pi_1))=0$ implies
$\tau_{f,2}(\overline \pi_2)\in \rm{Ker}(\overline\phi_2)$ or
$\tau_{f,2}(\overline \pi_1))\in \rm{Ker}(\overline\phi_1)$. Because
$d\varphi_N(\tau_{f,2}(\overline \pi_2))$ and
$d\varphi_M(\tau_{f,2}(\overline \pi_1))$ don't have the usual
 sense for differential map but only a kind of special notation, see already
stipulations (\ref{4.2.4-1b}).
 \end{remark}

If we interchange the roles between the domain and codomain of
$\overline \Phi$, we will obtain another type of product map such as
\[ \widehat \Psi=\widehat{\varphi_M \times \varphi_N}: (M\times N,g\oplus h) \to M\times_\lambda N \]
defined by $\widehat\Psi(x,y)=(\varphi_M(x), \varphi_N)$.

Under this case, although we finally expect that the operators $\bar
\nabla$ and $\bar R$ will be fully applied more than previous cases,
it is pity that $\tau_{f,2}(\widehat \Psi)$ is hard to work out.
More precisely, we hardly find a simple form for like the previous
cases. For instance, let $\varphi_M=Id_M$, then we can quickly get
   \[ \tau(\widehat\Psi)=-e(\varphi_N)(\rm{grad}_{g}\lambda^2,0_2),  \]
and
   \[ \tau_f(\widehat\Psi)
      =\big(-\,e(\varphi_N)f\,\rm{grad}_{g}\lambda^2+ \rm{grad}_{g}f, d\varphi_N(\rm{grad}_h f\big),\]
where $e(\varphi_N)$ is the energy density of $\varphi_N$,
$e(Id_N)=\frac{n}{2}$. Next, we involve to tackle the terms such as
  \[\sum\limits_{j=1}^m \big(\bar\nabla_{(e_j,0_2)}\bar\nabla_{(e_j,0_2)}
         -\bar\nabla_{{}^M\!\nabla_{e_j} e_j}
         \big)\tau_f(\widehat\Psi),    \]
  \[\sum\limits_{\alpha=1}^n \big(\bar\nabla_{(0_1,d\varphi_N(\bar e_\alpha)}\bar\nabla_{(0_1,d\varphi_N(\bar e_\alpha)}
         -\bar\nabla_{{}^N\!\nabla_{d\varphi_N(\bar e_\alpha)} d\varphi_N(\bar e_\alpha)}
         \big)\tau_f(\widehat\Psi),    \]
\[\sum\limits_{j=1}^m \bar R\big((0_1,e_j),\tau_f(\widehat\Psi)\big)(0_1,e_j)  \]
\[\sum\limits_{\alpha=1}^n \bar R\big((0_1,d\varphi_N(\bar e_\alpha)),\tau_f(\widehat\Psi)\big)
    (0_1,d\varphi_N(\bar e_\alpha).  \]
This produces much more sub-terms which are not good to integral.
Based on the disadvantage, we omit investigating the product map
$\widehat\Psi$.


\section{The behaviors of $f$-bi-harmonic maps from or into singly WPM }

In this section, we will discuss the behavior of $f$-bi-harmonicity
combining with singly WPM like the previous section.

\subsection{$f$-Bi-harmonicity of the inclusion maps}

The goal of this subsection is to characterize the
$f$-bi-harmonicity of the inclusion map $i_{x_0}:(N,h) \to
M\times_{\lambda} N$ ($x_0 \in M$) in terms of warping function
$\lambda$. As for the inclusion $i_{y_0} : (M,g) \to  (M
\times_\lambda N $ ($y_0\in N$), since it is always a totally
geodesic map, it is harmonic and $f$-bi-harmonic for any warping
function $\lambda$. We have

 \begin{proposition} \label{5-17-1}
Let $f: N \to (0,+\infty)$ be a smooth function. For the inclusion
map $i_{x_0}:(N,h) \to M\times_{\lambda} N$, its $f$-bi-tension
fields is given by
   \begin{equation} \label{5-19-11a}
  \begin{array}{ll}
    \tau_{f,2}(i_{x_0})
  &=\big(-\frac{n^2}{8}f(y)\big(\mathrm{grad}_g (|\mathrm{grad}_g \lambda^2|^2)
      + \frac{n}{2} (\Delta_N f(y))\mathrm{grad}_g \lambda^2,0_2\big)
      \\& \qquad +\big(0_1, 2n|\rm{grad}_g\lambda|^2\rm{grad}_h f(y)\big)\big|_{i_{x_0}}.
   \end{array}
    \end{equation}
\end{proposition}

  \begin{proof} Refer to \cite{BMO}. Let $\{\bar e_\alpha\}_{\alpha=1}^n$ be an orthonormal frame
  $N$. From (\ref{5-19-4}), we have
   \begin{equation} \notag
  \begin{array}{ll}
\tau(i_{x_0})=-\frac{n}{2}(\mathrm{grad}_g\lambda^2,0_2)\mid_{i_{x_0}}.
\end{array}
 \end{equation}
It is clear that $i_{x_0}$ is harmonic if and only if
$\mathrm{grad}_h\mu^2\mid_{i_{x_0}} = 0$.

Further similar to (\ref{5-19-5a}), (\ref{5-19-5b}) and
(\ref{5-19-5c}), we have

\begin{equation}
\begin{split} \label{5-19-10a}
   &(\nabla^{i_{x_0}}_{\rm{grad}_h f(y)} \tau_f(i_{x_0})
    =-n|\rm{grad}_g\lambda|^2\big(0_1, \rm{grad}_h f(y) \big)\big|_{i_{x_0}},
 \end{split}
 \end{equation}
 \vskip -6mm
\begin{equation}
\begin{split} \label{5-19-10b}
   &\rm{Tr}_h (\nabla^{i_{x_0}})^2 \tau_f(i_{x_0})
    =frac{n^2}{2}|\rm{grad}_g\lambda|^2(\rm{grad}_{g} \lambda^2,0_2)\big|_{i_{x_0}},
 \end{split}
 \end{equation}
\vskip -6mm
     \begin{equation} \label{5-19-10c}
      \begin{array}{ll}
   & \sum\limits_{\alpha=1}^n\bar R\big(\tau(i_{x_0 }),(0_1,\bar e_\alpha)\big)(0_1,\bar e_\alpha)
  \\ & =\frac{n^2}{4}\lambda  \left(\,^{M}\!\nabla_{\mathrm{grad}_g \lambda^2} \mathrm{grad}_g \lambda^2
             -\frac{1}{2\lambda^2} \mathrm{grad}_g \lambda^2(\lambda^2)\mathrm{grad}\ \lambda^2,  0_2\right)
  \\ & = \frac{n^2}{8}\big(\mathrm{grad}_g (|\mathrm{grad}_g \lambda^2|^2),0_2\big)
        -\frac{n^2}{2}|\mathrm{grad}_g \lambda |^2 (\mathrm{grad}_g \lambda^2,0_2).
   \end{array}
 \end{equation}
Combining these equations with (\ref{5-12-2}), we obtain
 \begin{equation} \label{5-19-11}
  \begin{array}{ll}
    \tau_{f,2}(i_{x_0})
  & =-\frac{n^2}{8}f(y)\big(\mathrm{grad}_g (|\mathrm{grad}_g \lambda^2|^2),0_2\big)
     +\frac{n}{2} (\Delta_N f(y))(\mathrm{grad}_g \lambda^2,0_2)
 \\ & \qquad +2n|\rm{grad}_g\lambda|^2\big(0_1, \rm{grad}_h f(y)\big)\big|_{i_{x_0}}
   \\ &=\big(-\frac{n^2}{8}f(y)\big(\mathrm{grad}_g (|\mathrm{grad}_g \lambda^2|^2)
      + \frac{n}{2} (\Delta_N f(y))\mathrm{grad}_g \lambda^2,0_2\big)
      \\& \qquad +\big(0_1, 2n|\rm{grad}_g\lambda|^2\rm{grad}_h
      f(y)\big)\big|_{i_{x_0}},
   \end{array}
    \end{equation}
as claimed.
 \end{proof}

\begin{corollary} \label{5-19-12a}
  Let $f:(N,h) \to (0,+\infty)$ be a smooth function. The inclusion map $i_{x_0}: (N,h) \to M\times_\lambda N$
 is a $f$-bi-harmonic map if and only if $\lambda$ and $f$ satisfy
  \begin{equation} \label{5-19-7}
   \left\{ \begin{array}{ll}
      & nf(y)\big(\mathrm{grad}_g (|\mathrm{grad}_g \lambda^2|^2)
      -8(\Delta_N f(y))\mathrm{grad}_g \lambda^2 \big|_{i_{x_0}}=0,
      \\ &  |\rm{grad}_g\lambda|^2\rm{grad}_h
      f(y)\big)\big|_{i_{x_0}}=0.
   \end{array}
   \right.
    \end{equation}
\end{corollary}

From the second equation in (\ref{5-19-7}), we easily observe that
it holds if and only if either $\lambda\circ i_{x_0}$ or $f\circ
i_{x_0}$ is constat. This implies that
\begin{corollary} \label{5-19-12b}
 Let $f: N \to (0,+\infty)$ be a smooth function. The inclusion map
 $i_{x_0}: (N,h) \to M\times_\lambda N$ admits no a non-trivial $f$-bi-harmonic map.
\end{corollary}

\begin{remark} \label{5-19-12c} If in (\ref{5-19-11a}), $\tau_{2,f}(i_{x_0})$ only contains the first term
on the right-hand side, then when $x_0$ is a critical point of
$|\mathrm{grad}_g \lambda^2|^2$ but not a critical point of
$\lambda^2$, the inclusion map $i_{x_0}$  is a non-trivial
$f$-bi-harmonic map. For example (c.f. \cite{CMO1,CMO2}), let
 $S^n$ be a unit Euclidean sphere with dimension $n$. Then for $p\in S^{n+1}$, the space $S^{n+1}-\{\pm
p\}$ can be viewed as the SWPM
    \[(0,\pi) \times_{\sin t} S^n.\]
Consider the inclusion map
    \[ i_{\pi/4} (resp.\ i_{3\pi/4}): S^n \to 0,\pi) \times_{\sin t} S^n.\]
Since $\rm{grad}_t \sin^2 t \big|_{ \pi/4}=\rm{grad}_t
\sin^2t\big|_{ 3\pi/4}=0$ but $\sin^2(\pi/4)=1/2=\sin^2(3\pi/4)$, by
Corollary \ref{5-19-12b}, we know that $i_{\pi/4}$ and $i_{3\pi/4}$
are non-trivial $f$-bi-harmonic map with non-constant positive
function $f(y)|_{S^n}$.
\end{remark}

\subsection{$f$-Bi-harmonicity of the projection maps}

From Subsection \ref{6.2bi-f}, we have known that the second factor
projection map $\bar\pi_2 :M\times_\lambda N \to N$ satisfies
$\tau(\bar\pi_2)=0$, whereas the first factor projection map
$\bar\pi_1 :M\times_\lambda N \to M$ satisfies $\tau(\bar\pi_1)\neq
0$ with non-constant positive function $f$. Thus $\bar \pi_2$ is a
trivial case since it automatically becomes a $f$-bi-harmonic map.
Now, we only need to consider the projection map $\bar\pi_1$.

 \begin{lemma} \label{5-19-13}
Let $f: M\times_\lambda N \to (0,+\infty)$ be smooth function. The
$f$-bi-tension field of $\bar\pi_1: M \times_\lambda N \to M$ is
given by
     \begin{equation} \label{5-19-13a}
      \begin{array}{ll}
   \tau_{f,2}(\overline\pi_1) & =-nf(x,y) \rm{Tr}_g (\,^M\!\nabla)^2 \rm{grad}_{ g} \log \lambda)
             -  \frac{n^2}{2}f(x,y) \,^M\! \rm{grad}_  g(|\rm{grad}_g \log \lambda|^2)
       \\ &   -n f(x,y){}^M\!\rm{Ric}(\rm{grad}_{ g}\log \lambda )
               -n (\Delta_{M\times_\lambda N} f(x,y))\rm{grad}_{ g}\log \lambda
        \\ & \quad        -2n\,^M\!\nabla_{\rm{grad}_{g} f } \rm{grad}_{g} \log \lambda)\big| \bar
        \pi_1.
     \end{array}
     \end{equation}
\end{lemma}

\begin{proof} Similar to the proof of Lemma \ref{4.2.3-01},
we have
    \begin{equation} \notag
  \begin{array}{ll}
   & \tau(\bar\pi_1)=n \rm{grad}_{g} \log\lambda \mid \bar\pi_1,
  \end{array}
 \end{equation}

    \begin{equation} \notag
  \begin{array}{ll}
   & \nabla^{\bar\pi_1}_{\rm{grad}_{\bar g} f(x,y) } \tau(\bar\pi_1)
       =n \,^M\!\nabla_{\rm{grad}_{g} f } \rm{grad}_{g} \log \lambda)\circ \bar \pi_1,
     \end{array}
 \end{equation}

    \begin{equation} \notag
    \begin{array}{ll}
    Tr_{\bar g}(\nabla^{\bar\pi_1})^2 \tau(\bar\pi_1)
    & =n Tr_g (\,^M\!\nabla)^2 \rm{grad}_{g} \log(\lambda)\\
        & \hskip 2cm +\frac{n^2}{2} \,^M\! \rm{grad}_  g
          (|\rm{grad}_g \log \lambda|^2)\circ \bar \pi_1,
     \end{array}
 \end{equation}

 \begin{equation} \notag
  \begin{array}{ll}
    Tr_{\bar g} R^M(\tau(\bar\pi_1),d\bar \pi_1 )d\bar\pi_1
   & =  n {}^M\!\rm{Ric}(\rm{grad}_{ g}\log \lambda )\circ \bar
   \pi_1,
     \end{array}
    \end{equation}
from which, we get
     \begin{equation} \notag
      \begin{array}{ll}
   \tau_{f,2}(\overline\pi_1)&=-f\big( \rm{Tr}_{\bar g}(\nabla^{\bar\pi_1})^2 \tau(\bar\pi_1)
         -f \rm{Tr}_{\bar g} R^M(\tau(\bar\pi_1),d\bar \pi_1) d\bar\pi_1 \big)
    \\ & \quad      -(\Delta_{M\times_\lambda N} f(x,y))\tau(\bar\pi_1)-2 \nabla^{\bar\pi_1}_{\rm{grad}_{\bar g} f(x,y) } \tau(\bar\pi_1)
    \\   & =-nf(x,y) \rm{Tr}_g (\,^M\!\nabla)^2 \rm{grad}_{ g} \log \lambda)
             -  \frac{n^2}{2}f(x,y) \,^M\! \rm{grad}_  g(|\rm{grad}_g \log \lambda|^2)
       \\ &   -n f(x,y){}^M\!\rm{Ric}(\rm{grad}_{ g}\log \lambda )
               -n (\Delta_{M\times_\lambda N} f(x,y))\rm{grad}_{ g}\log \lambda
        \\ & \quad        -2n\,^M\!\nabla_{\rm{grad}_{g} f } \rm{grad}_{g} \log \lambda)\big| \bar \pi_1
     \end{array}
     \end{equation}
Thus we complete the proof.
 \end{proof}

As a consequence, we have

 \begin{proposition} \label{4.2.3-1ab}
 The projection map $\bar\pi_1$ is a non-trivial $f$-bi-harmonic
map if and only if $\lambda$ and $f$ satisfy
    \begin{equation} \label{4.2.3-2aa}
      \begin{array}{ll}
    & f(x,y) \rm{Tr}_g (\,^M\!\nabla)^2 \rm{grad}_{ g} \log \lambda)
             + \frac{n}{2}f(x,y) \,^M\! \rm{grad}_  g(|\rm{grad}_g \log \lambda|^2)
       \\ &   + f(x,y){}^M\!\rm{Ric}(\rm{grad}_{ g}\log \lambda )
               +(\Delta_{M\times_\lambda N} f(x,y))\rm{grad}_{ g}\log \lambda
        \\ & \quad        +2\,^M\!\nabla_{\rm{grad}_{g} f } \rm{grad}_{g} \log \lambda)\big| \bar
        \pi_1 =0.
  \end{array}
   \end{equation}
  \end{proposition}

\subsection{$f$-Bi-harmonicity of the product maps with harmonic factor}

Now, we turn to consider two types of product map such as
\begin{equation} \label{5-20-1a}
\begin{array}{c}
  \Psi=Id_M \times \psi:  M \times_\lambda N
 \to (M\times N, g\oplus h)\\
 \Psi(x, y) = (x,\psi(y))
 \end{array}
 \end{equation}
 and
   \begin{equation} \label{5-20-1b}
\begin{array}{c}
  \overline\Psi=\overline{Id_M \times \psi}:  (M\times N, g\oplus h) \to M \times_\lambda N
 \\ \overline\Psi(x, y) = (x,\psi(y))
 \end{array}
 \end{equation}
 where $\psi: N \to N$ is a harmonic map.

For the first case (\ref{5-20-1a}), we have
 \begin{proposition} \label{5-20-2}
 Suppose that $\psi: N \to N$ is a harmonic map and $ f\in C^\infty(M \times_\lambda N)$ is a
 positive function. For the product map $\Psi=Id_M \times \psi:  M \times_\lambda N
 \to (M\times N, g\oplus h)$, we have
      \begin{equation} \label{5-20-2a}
    \begin{array}{ll}
   \tau_{2,f}(\Psi)&= (\tau_{2,f}(\bar\pi_1), 0)
     .
\end{array}
 \end{equation}
\end{proposition}

\begin{proof} Similar to the proof of Proposition \ref{4.2.4-1}, by using
$\tau(\psi)=0$ we have
  \begin{equation} \label{5-20-3}
  \begin{array}{ll}
  \tau(\overline\Psi)&= \rm{Tr}_{\bar g} \nabla d\Psi
    \\&= \frac{n}{2\lambda^2}(\rm{grad}_g \lambda^2,0_2)+\frac{1}{\lambda^2}(0_1,\tau(\psi))\\
     &=\frac{n}{2\lambda^2}(\rm{grad}_g \lambda^2,0_2).
  \end{array}
 \end{equation}
Further, we get
   \begin{equation} \label{5-20-3a}
  \begin{array}{ll}
    \nabla^{\Psi}_{\rm{grad}_{\bar g} f} \tau (\bar\Phi)
    & =n\ \nabla_{(\rm{grad}_g f , \frac{1}{\lambda^2}\rm{grad}_h f )}
    (\rm{grad}_g \log \lambda,0_2)
    \\ & = n\ ( {}^M\!\nabla_{\rm{grad}_g  f } \rm{grad}_g \log \lambda, 0_2),
 \end{array}
 \end{equation}

\begin{equation} \label{5-20-3b}
  \begin{array}{ll}
    \rm{Tr}_{\bar g}(\nabla^{\bar\Phi})^2 \tau_f(\bar\Phi)&
    =\sum\limits_{j=1}^m \left(  \nabla_{(e_i,0_2)}\nabla_{(e_i,0_2)} \tau(\Psi)
     -\nabla_{d\Psi \big(\bar \nabla_{(e_i,0_2)}(e_i,0_2) \big)} \tau(\Psi) \right)
      \\ & +\frac{1}{\lambda^2}\sum\limits_{\alpha=1}^n
      \left(  \nabla_{(0_1,\bar {e}_\alpha)}\nabla_{(0_1,\bar {e}_\alpha)} \tau(\Psi)
     -\nabla_{d\Psi \big(\bar \nabla_{(0_1,\bar {e}_\alpha)}(0_1,\bar {e}_\alpha) \big)} \tau(\Psi) \right)
  \\ & = n \rm{Tr}_g (\,^M\!\nabla )^2 \rm{grad}_g \log(\lambda)
     +\frac{n^2}{2}\rm{grad}_g (|\rm{grad}_{g}\log \lambda|^2, 0_2),
  \end{array}
 \end{equation}

 \begin{equation} \label{5-20-3c}
\begin{array}{ll}
  \rm{Tr}_{\bar g} R\big(\tau_f(\Psi,d\Psi)d\Psi
     &=n \sum\limits_{j=1}^m  R \big(\rm{grad}_g \log(\lambda),0_2), (e_i,0_2)\big)(e_i,0_2)
     \\& +\frac{n}{\lambda^2}\sum\limits_{\alpha=1}^n R \big(\rm{grad}_g \log\lambda,0_2),
     (0_1,d\psi(\bar {e}_\alpha))\big)(0_1,d\psi(\bar {e}_\alpha))
     \\& =n({}^M\rm{Ric}(\rm{grad}_g \log\lambda)),0_2).
     \end{array}
    \end{equation}
Putting (\ref{5-20-3a}), (\ref{5-20-3b}) and (\ref{5-20-3c})
together, we have
      \begin{equation} \label{5-20-4}
      \begin{array}{ll}
   \tau_{2,f}(\Psi)&=-f\big( \rm{Tr}_{\bar g}(\nabla^{\bar\pi_1})^2 \tau(\Psi)
         -f \rm{Tr}_{\bar g} R^M(\tau(\Psi),d\bar \pi_1) d\Psi \big)
    \\ & \quad      -(\Delta_{M\times_\lambda N} f(x,y))\tau(\Psi)
      -2 \nabla^{\Psi}_{\rm{grad}_{\bar g} f(x,y) } \tau(\Psi)
    \\  & =-n \Big( f(x,y) \rm{Tr}_g (\,^M\!\nabla)^2 \rm{grad}_{ g} \log \lambda)
             +  \frac{n}{2}f(x,y) \rm{grad}_  g(|\rm{grad}_g \log \lambda|^2)
       \\ &   + f(x,y){}^M\!\rm{Ric}(\rm{grad}_{ g}\log \lambda )
               +(\Delta_{M\times_\lambda N} f(x,y))\rm{grad}_{ g}\log \lambda
        \\ & \quad        +\,^M\!\nabla_{\rm{grad}_{g} f } \rm{grad}_{g} \log \lambda),0_2\Big)
   .
     \end{array}
     \end{equation}
Finally, by using Lemma \ref{5-19-13}, (\ref{5-20-4}) can
 be written as
      \[  \tau_{2,f}(\Psi) =\big( \tau_{2,f}(\overline\pi_1), 0_2 \big), \]
as claimed.
 \end{proof}

As a consequence, we have

\begin{proposition} \label{5-20-5}
  Suppose that $\psi: N \to N$ is a harmonic map and $ f\in C^\infty(M \times_\lambda N)$ is a
 positive function. The product map $\Psi=Id_M \times \psi:  M \times_\lambda N
 \to (M\times N, g\oplus h)$ is a non-trivial $f$-bi-harmonic if and
 only if the warping function $\lambda$ is a non-constant solution
 to
      \begin{equation} \label{5-20-5a}
    \begin{array}{ll}
  & 0= f(x,y) \rm{Tr}_g (\,^M\!\nabla)^2 \rm{grad}_{ g} \log \lambda
             +  \frac{n}{2}f(x,y) \rm{grad}_  g(|\rm{grad}_g \log \lambda|^2)
       \\ & \hskip 12mm    + f(x,y){}^M\!\rm{Ric}(\rm{grad}_{ g}\log \lambda )
               +(\Delta_{M\times_\lambda N} f(x,y))\rm{grad}_{ g}\log \lambda
        \\ & \hskip 28mm       +\,^M\!\nabla_{\rm{grad}_{g} f } \rm{grad}_{g} \log \lambda.
\end{array}
 \end{equation}
\end{proposition}

 \begin{remark} It is easy to see that besides its harmonicity, the map $\psi$ gives no other
 contribution to the $f$-bi-tension field of $\Psi$. This enables us to construct a wide range
  of examples of non-trivial $f$-bi-harmonic maps of product type (\ref{5-20-1a}).
 \end{remark}

Next, we consider the second case (\ref{5-20-1b}), which just
interchanges the domain and codomain in the first case
(\ref{5-20-1a}). In this case, as we shall see, the information of
$\psi$ increases, involving its energy density $e(\phi)$. The
difficulties are due to the contribution of the curvature tensor
field $\bar R$ (see previous (\ref{5-6-3}) or (\ref{6-30-3})) of
$M\times_\lambda N$ in the expression of the $f$-bi-tension field.
We have

\begin{proposition} \label{5-20-6}
 Suppose that $\psi: N \to N$ is a harmonic map and $ f\in C^\infty(M \times N)$ is a
 positive function. For the product map
    $\overline\Psi=\overline{Id_M \times \psi}: (M\times N, g\oplus h) \to M \times_\lambda N $,
    its $f$-bi-tension field is given by
      \begin{equation} \label{5-20-6a}
    \begin{array}{ll}
   \tau_{2,f}(\overline \Psi)
     & =\Big(e(\psi)f(x,y) (\,\rm{Tr}_g ({}^M\!\nabla)^2\rm{grad}_g\lambda^2
      -\frac{1}{2}e^2(\psi)f(x,y) \rm{grad}_g|\rm{grad}_g \lambda^2|^2
   \\& +2 e(\psi)f(x,y)\,{}^M\rm{Ric}(\rm{grad}_g \lambda^2)
      +d\psi(\Delta_N e(\psi))f(x,y) \rm{grad}_g \lambda^2
   \\&+(\Delta_{M\times N} f(x,y)) e(\psi)\rm{grad}_g \lambda^2
      +e(\psi) \, {}^M\! \nabla_{\rm{grad}_g f(x,y) } \rm{grad}_g \lambda^2
   \\& \qquad +  2\rm{grad}_h f(x,y)(\ e(\psi)\,) \rm{grad}_g \lambda^2, 0_2\Big)
   \\& +\big(0_1, \frac{e(\psi)}{\lambda}|\rm{grad}_g \lambda^2|^2 \rm{grad}_h f(x,y)
      + \frac{1}{\lambda^2}|\rm{grad}_g \lambda^2|^2 f(x,y)d\psi(\rm{grad}_g e(\psi)\,\big )
   ,\end{array}
 \end{equation}
 where $d\psi(\Delta_N e(\psi) )$ and $d\psi(\rm{grad}_g e(\psi)\,)$
are defined by (\ref{5-20-7b1}) and (\ref{5-20-7b2}), respectively.
\end{proposition}

\begin{proof}  Let $\{e_j\}_{j=1}^m $  and $\{\bar e_\alpha\}_{\alpha=1}^n$
be orthonormal basises  on $(M,g)$ on $(N,h)$. Then\\
$\{(e_j,0_2),(0_1, \bar e_\alpha)\}_{j=1,\ldots,m,\atop
  \alpha=1,\ldots,n}$ is a local orthonormal basis on the direct product manifold $M\times N$.

Note that $\tau(\psi)=0$ and
$e(\psi)=\frac{1}{2}\frac{1}{\lambda^2}\sum\limits_{\alpha=1}^n
h\big(d\psi(\bar {e}_\alpha),d\psi(\bar {e}_\alpha) \big)$, we have
  \begin{equation} \label{5-20-7}
  \begin{array}{ll}
  \tau(\overline\Psi)&= \rm{Tr}_{ g\oplus h} \nabla d\overline\Psi
   \\ & =\sum\limits_{j=1}^m \big( \bar\nabla_{(e_j,0_2)}(e_j,0_2)
                 -(\nabla_{(e_j,0_2)}(e_j,0_2) \big)\\
      & +\sum\limits_{\alpha=1}^n \big(
      \bar\nabla_{(0_1,d\bar\psi(\bar e_\alpha)\,)}(0_1,d\bar\psi(\bar e_\alpha)\,)
       -(0_1,d\psi(\bar\nabla_{ e_\alpha} \bar e_\alpha\,)\,\big)\\
     & =(0_1, \tau(\psi)\,)- \frac{1}{2}\sum\limits_{\alpha=1}^n h(d\psi(\bar e_\alpha),d\psi(\bar e_\alpha)\,)
          (\mathrm{grad}_{g} \lambda^2,\,0_2)
          +(0_1, \bar\nabla_{\bar e_\alpha}\bar e_\alpha \,)\big)\\
     &=-e(\psi)(y)\,(\rm{grad}_g \lambda^2,0_2).
  \end{array}
 \end{equation}
 Further, we get
   \begin{equation} \label{5-20-7a}
  \begin{array}{ll}
    \nabla^{\overline\Psi}_{\rm{grad}_{ g\oplus h} f(x,y)} \tau (\overline\Psi)
        & =\bar \nabla_{(\rm{grad}_g f,\rm{grad}_h f )}
    -e(\psi)(y)(\rm{grad}_g \lambda^2,0_2)
    \\ & =- \ \rm{grad}_h f(\ e(\psi)(y)\,)(\rm{grad}_g \lambda^2,0_2)
    \\ &\hskip 6mm - e(\psi)(y) \bar \nabla_{(\rm{grad}_g f, \rm{grad}_g  f) } (\rm{grad}_g \lambda^2, 0_2)
    \\ &=- \ \rm{grad}_h f(\ e(\psi)\,)(\rm{grad}_g \lambda^2,0_2)
    \\& \hskip 3mm - e(\psi)(\, {}^M\! \nabla_{(\rm{grad}_g f } \rm{grad}_g \lambda^2, 0_2)
   -\frac{e(\psi)}{2\lambda}|\rm{grad}_g \lambda^2|^2(\,0_1,\rm{grad}_h f\,)
 \end{array}
 \end{equation}

\begin{equation} \label{5-20-7b}
  \begin{array}{ll}
    \rm{Tr}_{g\oplus h}(\nabla^{\overline\Psi})^2 \tau(\overline\Psi)&
    =\sum\limits_{i=1}^m \left( \bar \nabla_{(e_i,0_2)}\bar\nabla_{(e_i,0_2)} \tau(\overline\Psi)
        -\bar\nabla_{\nabla _{(e_i,0_2)}(e_i,0_2)}\right)\tau(\overline \Psi)
    \\& -\sum\limits_{\alpha=1}^n
    \left( \bar\nabla_{ (0_1,d\psi(\bar {e}_\alpha),)} \bar\nabla_{ (0_1,d\psi(\bar {e}_\alpha),)}
     -\bar\nabla_{\nabla_{(0_1,\bar {e}_\alpha)} (0_1, \bar {e}_\alpha) }\right)\tau(\overline \Psi)
  \\&  =-e(\phi) \left( \sum\limits_{i=1}^m \big({}^M\!\nabla_{e_i}{}^M\!\nabla_{e_i}
          -{}^M\!\nabla_{ {}^M\! \nabla_{e_i} e_i }\big) \rm{grad}_g \lambda^2, 0_2 \right)
      \\ & -\sum\limits_{\alpha=1}^n
      \left( d\psi(\bar {e}_\alpha),)d\psi(\bar {e}_\alpha)
      - d\psi({}^N\! \nabla_{\bar {e}_\alpha} \bar {e}_\alpha )(e(\psi)\right)
         (\rm{grad}_g \lambda^2,0_2)
     \\ & -\frac{e(\psi)}{2\lambda^2}|\rm{grad}_g \lambda^2|^2
     \big(\,0_1, \sum\limits_{\alpha=1}^n \big( {}^N\!\nabla_{d\bar\psi(\bar e_\alpha)} d\bar\psi(\bar e_\alpha)
      -d\psi({}^N\!\nabla_{ e_\alpha} \bar e_\alpha)\, \big)
    \\& - \frac{1}{\lambda^2}|\rm{grad}_g \lambda^2|^2
      \big(\, 0_1, \sum\limits_{\alpha=1}^n d\psi(\bar {e}_\alpha)(\,e(\psi)\,) d\psi(\bar {e}_\alpha)\,\big)
  \\ & = -e(\psi) (\,\rm{Tr}_g ({}^M\!)^2\rm{grad}_g \lambda^2,0_2\,)
       -d\psi(\Delta_N e(\psi)) (\rm{grad}_g \lambda^2,0_2)
   \\& -\frac{e(\psi)}{2\lambda^2}|\rm{grad}_g \lambda^2|^2(0_1,\tau(\psi)\,)
 - \frac{1}{\lambda^2}|\rm{grad}_g \lambda^2|^2 \big(0_1,d\psi(\rm{grad}_g e(\psi)\,) \big)
   \\&+ \frac{1}{2\lambda^2}|\rm{grad}_g \lambda^2|^2 e^2(\psi)(\rm{grad}_g \lambda^2,0_2),
    \end{array}
 \end{equation}
where
  \begin{equation} \label{5-20-7b1} d\psi(\Delta_N e(\psi) ) =\sum\limits_{\alpha=1}^n
      \left( d\psi(\bar {e}_\alpha),)d\psi(\bar {e}_\alpha)
      - d\psi({}^N\! \nabla_{\bar {e}_\alpha} \bar {e}_\alpha )(e(\psi)\right),
      \end{equation}
      \begin{equation} \label{5-20-7b2}
    d\psi(\rm{grad}_g e(\psi)\,)
       =\sum\limits_{\alpha=1}^n d\psi(\bar {e}_\alpha)(\,e(\psi)\,) d\psi(\bar {e}_\alpha),
    \end{equation}

 \begin{equation} \label{5-20-7c}
\begin{array}{ll}
  \rm{Tr}_{ g \oplus h} R\big(\tau(\overline\Psi,d\overline\Psi)d\overline\Psi
   & =-e(\psi) \sum\limits_{j=1}^m  \bar R \big( (\rm{grad}_g \lambda^2,0_2), (e_i,0_2)\big)(e_i,0_2)
     \\& -e(\psi)\sum\limits_{\alpha=1}^n \bar R \big(
      (\rm{grad}_g \lambda^2,0_2), (0_1,d\psi(\bar {e}_\alpha) \big)(0_1,d\psi(\bar {e}_\alpha),)
    \\ &=-e(\psi)(\,{}^M\rm{Ric}(\rm{grad}_g \lambda^2),0_2\,)
     \\& +\frac{1}{2}e^2(\psi) \big(\rm{grad}_g|\rm{grad}_g \lambda^2|^2,0_2 \big)
     \\& -\frac{1}{2\lambda^2}e^2(\psi)|\rm{grad}_g \lambda^2|^2 (\rm{grad}_g \lambda^2,0_2)
     \end{array}
    \end{equation}
Putting (\ref{5-20-7a}), (\ref{5-20-7b}) and (\ref{5-20-7c})
together, we have
      \begin{equation} \label{5-20-4}
      \begin{array}{ll}
   \tau_{f,2}(\overline\Psi)&=-f\big( \rm{Tr}_{g\oplus h}(\nabla^{\overline\Psi})^2 \tau(\overline\Psi)
         -f \rm{Tr}_{ g\oplus h } \bar R(\tau(\overline\Psi),d \overline\Psi) d\overline\Psi \big)
    \\ & \quad      -(\Delta_{M\times N} f(x,y))\tau(\overline\Psi)
      -2 \nabla^{\overline\Psi}_{\rm{grad}_{ g \oplus h} f(x,y) } \tau(\overline\Psi)
    \\  & =e(\psi)f(x,y) (\,\rm{Tr}_g ({}^M\!\nabla)^2\rm{grad}_g \lambda^2,0_2\,)
       +d\psi(\Delta_N e(\psi))f(x,y) (\rm{grad}_g \lambda^2,0_2)
      \\& + \frac{1}{\lambda^2}|\rm{grad}_g \lambda^2|^2 f(x,y)\big(0_1,d\psi(\rm{grad}_g e(\psi)\,) \big)
    \\& +e(\psi)f(x,y)(\,{}^M\rm{Ric}(\rm{grad}_g \lambda^2),0_2\,)
     \\& -\frac{1}{2}e^2(\psi)f(x,y) \big(\rm{grad}_g|\rm{grad}_g \lambda^2|^2,0_2 \big)
    \\& +(\Delta_{M\times N} f(x,y)) e(\psi)(\rm{grad}_g \lambda^2,0_2\,)
    \\& + \ 2\rm{grad}_h f(\ e(\psi)\,)(\rm{grad}_g \lambda^2,0_2)
    \\& \hskip 3mm +2 e(\psi)(\, {}^M\! \nabla_{\rm{grad}_g f } \rm{grad}_g \lambda^2, 0_2)
    +\frac{e(\psi)}{\lambda}|\rm{grad}_g \lambda^2|^2(\,0_1,\rm{grad}_h f\,)
   \\& =\Big(e(\psi)f(x,y) (\,\rm{Tr}_g ({}^M\!\nabla)^2\rm{grad}_g\lambda^2
      -\frac{1}{2}e^2(\psi)f(x,y) \rm{grad}_g|\rm{grad}_g \lambda^2|^2
   \\& +2 e(\psi)f(x,y)\,{}^M\rm{Ric}(\rm{grad}_g \lambda^2)
      +d\psi(\Delta_N e(\psi))f(x,y) \rm{grad}_g \lambda^2
   \\&+(\Delta_{M\times N} f(x,y)) e(\psi)\rm{grad}_g \lambda^2
      +e(\psi) \, {}^M\! \nabla_{\rm{grad}_g f } \rm{grad}_g \lambda^2
   \\& \qquad +  2\rm{grad}_h f(\ e(\psi)\,) \rm{grad}_g \lambda^2, 0_2\Big)
   \\& +\big(0_1, \frac{e(\psi)}{\lambda}|\rm{grad}_g \lambda^2|^2 \rm{grad}_h f(x,y)
       +\frac{1}{\lambda^2}|\rm{grad}_g \lambda^2|^2 f(x,y)d\psi(\rm{grad}_g e(\psi)\,\big )
    , \end{array}
     \end{equation}
as claimed.
 \end{proof}

As a consequence, we have

\begin{proposition} \label{5-21-1}
  Suppose that $\psi: N \to N$ is a harmonic map and $ f\in C^\infty(M \times N)$ is a
 positive function. The product map $\overline\Psi=\overline{Id_M \times \psi}: (M\times N, g\oplus h)
 \to  M \times_\lambda N$ is a non-trivial $f$-bi-harmonic if and
 only if the warping function $\lambda$ is a non-constant solution
 to
      \begin{equation} \label{5-21-1a}
    \begin{array}{ll}
  & e(\psi)f(x,y) (\,\rm{Tr}_g ({}^M\!\nabla)^2\rm{grad}_g\lambda^2
      -\frac{1}{2}e^2(\psi)f(x,y) \rm{grad}_g|\rm{grad}_g \lambda^2|^2
   \\& +2 e(\psi)f(x,y)\,{}^M\rm{Ric}(\rm{grad}_g \lambda^2)
      +d\psi(\Delta_N e(\psi))f(x,y) \rm{grad}_g \lambda^2
   \\&+(\Delta_{M\times N} f(x,y)) e(\psi)\rm{grad}_g \lambda^2
      +e(\psi) \, {}^M\! \nabla_{\rm{grad}_g f(x,y) } \rm{grad}_g \lambda^2
   \\& \qquad +  2\rm{grad}_h f(x,y)(\ e(\psi)\,) \rm{grad}_g \lambda^2=0,
\end{array}
 \end{equation}
 and
     \begin{equation} \label{5-21-1b}
    \begin{array}{ll}
  &  e(\psi)\lambda \rm{grad}_h f(x,y) + f(x,y)d\psi(\rm{grad}_g e(\psi)) =0.
\end{array}
 \end{equation}
\end{proposition}


\section{Comparisons on bi-$f$-harmonic maps and $f$-bi-harmonic maps}

From the previous sections, we know that bi-$f$-harmonic map and
$f$-bi-harmonic map are two wider generalizations via mixing
bi-harmonic map and $f$-harmonic map. More precisely, according to
their tension fields, there are some following relations:
\\(1) $f$-harmonic map must be bi-$f$-harmonic map, but the usual
harmonic map is not bi-$f$-harmonic map. In particular,
bi-$f$-harmonic has nothing to do with bi-harmonic map.
\\(2) The usual harmonic map must be $f$-bi-harmonic, but
bi-harmonic map is not be $f$-bi-harmonic except $f=const.$. In
particular, $f$-bi-harmonic has nothing to do with $f$-harmonic map.
 Under the conformal dilation, from Corollaries \ref{3-31-3a} and
 \ref{3-31-3} we find that $Id_M$ is $f$-bi-harmonic map but not
 bi-$f$-harmonic map. The most difference between them is in that for the product map
    $\overline\Psi=\overline{Id_M \times \psi}: (M\times N, g\oplus h) \to M \times_\lambda N $,
  since the difficulties are due to the contribution of the curvature tensor field $\bar R$ (see
previous (\ref{5-6-3}) or (\ref{6-30-3})) of $M\times_\lambda N$ in
the expression of the $f$-bi-tension field and bi-$f$-tension field,
unlike Proposition \ref{5-20-6}, we have to give up acquiring the
bi-$f$-tension field (see the analysis in the last paragraph of
subsection 5.3.

So far, although bi-$f$-harmonic maps has some progress, see
\cite{CET, Ch1,Ch2}, $f$-bi-harmonic map is just an up-and-coming
thing. Do they contain inspired anticipations of the shape of things
to come? Or are there any promising outlook? We shall look forward
to their evolution in the future.


\begin{thebibliography}{99}

 \bibitem[BFO]{BFO} P. Baird, A. Fardoun And S. Ouakkas, \emph{Conformal and semi-conformal biharmonic maps},
annals of global analysis and geometry, 34(2008), 403-414.

\bibitem[BG]{BG} J\,K\,Beem, T\,G\,Powell. \emph{Geodesic completeness and
   maximality in Lorentzian warped products}, Tensor\,(N\,S), 1982, 39:31-36.

\bibitem[BMO]{BMO} A. Balmus, S. Montaldo and C. Oniciuc, {\em Biharmonic maps between warped product
    manifolds}, J. Geom. Phys. 57(2007), 449-466.

\bibitem[BW]{BW} P.~Baird and T.~C.~Wood, \emph{Harmonic Morphisms Between Riemannian Manifolds},
  Lond. Math. Soc. Monogr., New Series 29, Oxford University Press, 2003

\bibitem[CET]{CET} A. M. Cherif, H. Elhendi AND M. Terbeche, \emph{On generalized conformal maps},
   Bulletin of Mathematical Analysis and Applications, 4(4)(2012), 99-108.

\bibitem[Ch1]{Ch1} Yuan-Jen Chiang, \emph{f-biharmonic maps between Riemannian manifolds}, Journal
of Geometry and Symmetry in Physics, 27(2012), 45-58.

\bibitem[Ch2]{Ch2} Yuan-Jen Chiang, \emph{Transversally $f$-harmonic and transversally $f$-biharmonic
maps between foliated  manifolds}, JP Journal of Geometry and
topology, 13(1)(2012), 93-117. (can't loadown)

\bibitem[CMO1]{CMO1} R.~Caddeo, S.~Montaldo and C.~Oniciuc, \emph{Biharmonic submanifolds of
$S^3$}, Internat. J. Math., 12(2001), 867-876.

\bibitem[CMO2]{CMO2} R.~Caddeo, S.~Montaldo and C.~Oniciuc, \emph{Biharmonic submanifolds in spheres},
Israel. J. Math., 130(2002), 109-123.

\bibitem[Co]{Co} N. Course, {\em f-harmonic maps}, Ph.D.Thesis, University of Warwick, Coventry, CV4 7AL,
UK, 2004.

\bibitem[ES]{ES} J. Eells and J.~H.~ Sampson,{\em Harmonic mappings of Riemannian manifolds}, Amer. J.
            Math., 86(1964), 109-160.

\bibitem[Ji]{Ji}  G.~Y.~Jiang, {\em 2-harmonic maps and their first and second variation formulas}, Chinese
Ann. Math. Ser. A, 7(1986), 389-402.
\bibitem[L]{L} A. Lichnerowicz, {\em Applications harmoniques et varietes kahleriennes}, Symposia Mathematica
        III, Academic Press, London, 1970, 341-402.

\bibitem[Lu1]{Lu1} Weijun Lu, {\em $f$-Harmonic maps between doubly twisted product manifolds}, 2011-5-12
complete, Applied Mathematics---A Journal of Chinese Universities,
28(2)2013, 240-252

\bibitem[Lu2]{Lu2} Wei-Jun Lu, {\em Bi-$f$-Harmonic maps between doubly warped product manifolds},
 July 2011, http://www.cms.zju.edu.cn/news.asp?id=1588\&ColumnName=pdfbook\&Version=english,
  N0.11015

\bibitem[Lu3]{Lu3} Wei-Jun Lu, {\em Geometry of warped product manifolds and its five
applications}, Ph.D. thesis, Zhejiang University, 2013.

\bibitem[LW]{LW} Y. X. Li and Y. D. Wang, {\em Bubbling location for f-harmonic maps and inhomogeneous
       Landau-Lifshitz equations}, Comment. Math. Helv., 81(2)(2006), 433-448.

\bibitem[OND]{OND} S. Ouakkas, R. Nasri, and M. Djaa, {\em On the f-harmonic and f-biharmonic maps}, JP J.
    Geom. Topol. 10(1)(2010), 11-27.

\bibitem[Ou1]{Ou1} Ye-Lin Ou, {\em On $f$-harmonic morphisms between Riemannian manifolds},
      arXiv:1103.5687v1, 2011.

\bibitem[Ou2]{Ou2} Ye-Lin Ou, {\em On conformal bi-harmonic  immersions}, Ann. Glob. Anal. Geom.,
36(2009),133-142.

\bibitem[PK]{PK} S.~Y.~Perkta\c{s} and E.~Kili\c{c}, {\em Biharmonic maps between doubly warped product
 manifolds}, Balkan Journal of Geometry and Its Applications, (2)15(2010), 151-162.

\bibitem[RV]{RV} M.~Rimoldi and G.~Veronelli, \emph{f-Harmonic maps and applications
to gradient Ricci solitons}, arXiv: 1112.3637v1, 2011.

\bibitem[Un]{Un} B. Unal, {\em Doubly Warped Products}, Differential Geometry and Its Applications,
15(3)(2001), 253-263.


\end{thebibliography}
\end{document}